\documentclass[
a4paper,
]{scrartcl}

\usepackage[
theoremdefs,
notex4ht,
final,
]{latexdev}
\usepackage{unixode}

\usepackage[numbers]{natbib}
\usepackage{url}

\usetikzlibrary{matrix}
\tikzstyle{commdiag}=[matrix of math nodes, row sep=3em, column sep=2.5em, text height=1.5ex, text depth=0.25ex,ampersand replacement=\&]
\tikzset{>=stealth}

\usepackage{booktabs}
\usepackage{longtable}
\usepackage{multirow}

\usepackage{etrees}

\colorlet{emphcolor}{FireBrick}
\tikzset{treeemph/.style={draw=emphcolor, every node/.style={dtree black node, fill=emphcolor}}}

\newcommand*\ATb{\begin{tikzpicture}[setree]\placeroots{1}\end{tikzpicture}}
\newcommand*\ATbb{\begin{tikzpicture}[setree]\placeroots{1}\children{child{node{}}}\end{tikzpicture}}
\newcommand*\ATbapab[1][]{\begin{tikzpicture}[setree, #1]\placeroots{2}\jointrees{1}{1}\end{tikzpicture}}

\newcommand*\demph[1]{\textbf{#1}}

\newcommand*\Ra{{\mathcal R}}
\newcommand*\RR{\mathbf{R}}
\newcommand*\NN{\mathbf{N}}

\newcommand*\Man{\mathcal{M}}
\newcommand*\Rd{\RR^d}
\newcommand*\Tan[1][]{\mathsf{T}_{#1}}

\NewDocumentCommand\diff{O{\Man}}{\mathfrak{X}\IfNoValueTF{#1}{(\Man)}{(#1)}}
\NewDocumentCommand\Diff{O{\Man}}{\group{Diff}\IfNoValueTF{#1}{(\Man)}{(#1)}}

\newcommand*\eldiff{\mathcal{F}}

\newcommand*\Cinf{\mathcal{C}^{\infty}}
\newcommand*\Cinfloc{\locmark{\mathcal{C}}^{\infty}}
\newcommand*\Lin{\mathcal{L}}
\usepackage{accents}
\newcommand*\locmark[1]{\accentset{\bullet}{#1}}
\newcommand*\Linloc{\locmark{\mathcal{L}}}
\NewDocumentCommand\linloc{O{\vecbndl}O{\vecbndl[F]}}{%
	\Linloc_{\symgrp}\paren[\big]{\sym^m \diff[\Rd], \diff[\Rd]}
		}

\newcommand*\group[1]{\mathsf{#1}}
\newcommand*\Aff{\group{Aff}(d)}
\newcommand*\GL[1][d]{\group{GL}(#1)}
\newcommand*\SO[1][d]{\group{SO}(#1)}
\newcommand*\SE[1][d]{\group{SE}(#1)}

\newcommand*\Ebv{\boldsymbol{E}}
\newcommand*\Ev{M}
\newcommand*\Fbv{\boldsymbol{F}}
\newcommand*\Fv{M}
\newcommand*\Mv{M}

\newcommand*\vecbndl[1][E]{\mathcal{#1}}
\newcommand*\bndlproj[2][\Man]{#1 \leftarrow #2}
\newcommand*\sects[2][\Man]{\Cinf(\bndlproj[#1]{#2})}
\newcommand*\fibre[2]{#2_#1}

\newcommand*\Meth{\Phi}
\newcommand*\meth{\varphi}

\newcommand*\symgrp{\group{G}}
\newcommand*\isogrp{\group{H}}

\newcommand*\one{\mathsf{1}}

\newcommand*\act{\cdot}

\newcommand*\mat{\mathsf{A}}
\newcommand*\trans{\mathsf{b}}

\newcommand*\Div{\operatorname{div}}
\newcommand*\Tr{\operatorname{Tr}}

\newcommand*\dd{\mathrm{d}}
\newcommand*\Der{\operatorname{D}}
\newcommand*\supp{\operatorname{supp}}

\newcommand*\inv{^{-1}}

\newcommand*\origin{\mathsf{o}}
\newcommand*\Jetx[2][k]{J^{#1}_{#2}{\vecbndl}}
\newcommand*\Jet[1][k]{\Jetx[#1]{}}
\newcommand*\jet[1][\origin]{\operatorname{j}_{#1}^k}
\newcommand*\lfunc[1][\origin]{\overline\func_{#1}}
\newcommand*\facfunc[2][\origin]{\lfunc[#1]\paren[\big]{\jet[#1](#2)}}

\RenewDocumentCommand\Vec{O{\prt}}{\mathcal{V}_{#1}}

\NewDocumentCommand\Sym{O{\prt}}{\mathcal{S}_{#1}}
\NewDocumentCommand\Ten{O{\prt}}{\mathcal{T}_{#1}}
\NewDocumentCommand\tenexpr{O{\Mv}O{\Mv}O{\sym}}{#2 \otimes \bigotimes_{j=0}^{\infty} #3^{\prt(j)}(#1^* \otimes #3^j \Mv)}

\NewDocumentCommand\Dit{O{\prt}}{\Diag[#1]^1}
\NewDocumentCommand\Diag{O{\prt}}{\Gamma_{#1}}

\NewDocumentCommand\diag{}{\gamma}
\NewDocumentCommand\dit{}{\gamma}

\newcommand{\func}{\varphi}

\newcommand*\sym{{S}}
\newcommand*\ten{{T}}
\newcommand*\prt{\kappa}
\NewDocumentCommand\Prt{O{m}}{K_{#1}}

\usepackage{authblk}

\title{Aromatic Butcher Series}
\author{Hans Munthe-Kaas}
\author{Olivier Verdier}
\affil{Department of Mathematics, University of Bergen, Norway}

\begin{document}

\maketitle


\begin{abstract}
	We show that without other further assumption than affine equivariance and locality, a numerical integrator has an expansion in a generalized form of Butcher series (B-series) which we call \emph{aromatic B-series}.
	We obtain an explicit description of aromatic B-series in terms of elementary differentials associated to \emph{aromatic trees}, which are directed graphs generalizing trees.
%
%
	We also define a new class of integrators, the class of \emph{aromatic Runge--Kutta methods}, that extends the class of Runge--Kutta methods, and have aromatic B-series expansion but are not B-series methods.
	Finally, those results are partially extended to the case of more general affine group equivariance.
\end{abstract}

\begin{description}
	\item[Keywords] B-Series · Butcher series · Equivariance · Aromatic series · Aromatic trees · Functional graph · Directed pseudo-forest
\item[Mathematics Subject Classification (2010)]  37C80 · 37C10 · 41A58 · 15A72
\end{description}

\tableofcontents

\section{Introduction}

Numerical integration of differential equations is the art of approximating the exponential map, sending vector fields to their flows. A fundamental property of the exponential map is its \emph{equivariance} with respect to the full group of diffeomorphisms on the domain. This means that if a vector field is transformed by a diffeomorphism and afterwards exponentiated, we obtain exactly the same result as if the original vector field is exponentiated and the result is transformed by the  given  diffeomorphism. This complete equivariance of the exponential map is not shared by any approximation.

We are in particular interested in a class of methods called \emph{B-series methods} \cite[\S\,III.1.2]{HaLuWa06}.
These methods are characterized by a special kind of expansion.
In order to describe the expansion,
we take the point of view of backward error analysis 
\cite[\S\,IX]{HaLuWa06}.
The idea is that an integrators solves exactly the flow a modified version $\tilde{f}$ of the original vector field $f$.
B-series methods have the property that this modified vector field $\tilde{f}$ can be expanded in a specific form:
\begin{align}
	\label{eq:treeexp}
	\tilde{f} = b_0 f + b_1 f'(f) + b_2 f''(f,f) + b_3 f'(f'(f)) + \cdots
	,
\end{align}
where the terms are indexed by rooted trees \cite[\S\,III.1.2]{HaLuWa06}, \cite{Ha94}.
The rooted trees corresponding to the terms in \eqref{eq:treeexp} are
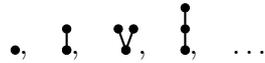
\begin{align}
	\ATb, \quad \ATbb, \quad 
	\begin{tikzpicture}[setree]
		\placeroots{1}
		\children{child{node{}} child{node{}}}
	\end{tikzpicture},\quad
\begin{tikzpicture}[setree]
	\placeroots{1}
	\children[1]{child{node{} child{node{}}}}
\end{tikzpicture}
,
\quad
\ldots
\end{align}
The class of B-series methods was defined after realising that Runge--Kutta methods have such an expansion in \cite{Bu72}.

McLachlan and Quispel observed that any Runge--Kutta method is equivariant with respect to the full affine linear group acting upon $\RR^n$ \cite{MLQu01}. McLachlan further observed that this holds for B-series in general, and he asked the natural question: \emph{"Can any affine equivariant method be expanded in a B-series?"}. This has been an important open question for many years. 

We answer the question by showing that the local and affine equivariant methods, in some fixed dimension, have an expansion which, however, turns out to be \emph{more general} than B-series expansion.
The first few terms of the expansion of a local, affine equivariant method, are of the form
\begin{align}
	\label{eq:aromexpansion}
	\meth(f) = b_0 f + b_1 f'(f) + b_2{\color{emphcolor}  \Div(f) f} + b_3 f''(f,f) + b_4 {\color{emphcolor} (\operatorname{grad} \Div(f), f) f}\cdots
\end{align}
By comparing with \eqref{eq:treeexp}, one observes that terms of a new kind appear, such as $ \Div(f) f$.
We are able to completely describe those terms, and they turn out to be
associated to \emph{aromatic trees}, which generalize rooted trees.
The aromatic trees corresponding to the terms in \eqref{eq:aromexpansion} are the following, where the new term are emphasized:
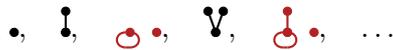
\begin{align}
	\ATb, \quad \ATbb, \quad \ATbapab[treeemph], \quad
	\begin{tikzpicture}[setree]
		\placeroots{1}
		\children{child{node{}} child{node{}}}
	\end{tikzpicture},\quad
\begin{tikzpicture}[setree, treeemph]
	\placeroots{2}
	\children[1]{child{node{}}}
	\jointrees{1}{1}
\end{tikzpicture}
,
\quad
\ldots
\end{align}
In \autoref{sec:introaffequiseries} we give a detailed description of aromatic trees and of the corresponding terms appearing in \eqref{eq:aromexpansion}.
In \autoref{sec:equirk} we also give some concrete examples of such methods, which are local, affine equivariant but not B-series. 
In \autoref{sec:introaffequiseries}, we also give an example of a theoretical, non-local and affine equivariant method, which does not have any expansion at all.

The traditional approach, for B-series and other similar series such as P-series, has been to focus on one particular family of methods, such as Runge--Kutta \cite{Bu72}, Rosenbrock methods \cite[\S\,IV.7]{HaWa10} or the average vector field methods \cite{QuML08}, and to show by direct calculations that they have a B-series expansion.
The biggest novelty in this paper is that we are able to find a series expansion for a class of methods \emph{which is not described explicitly}.
We only assume two very simple criteria, namely locality and affine equivariance, which are straightforward to check, and those immediately ensure the existence of a series expansion.

We introduce the following new ideas in the paper.

In \autoref{sec:equivariance}, we give a general definition of equivariance for numerical methods and modified vector fields which extends the equivariance concept of \cite{MLQu01}.
In particular, the property of affine equivariance means that the method is immune to affine coordinate transformations.
In other words, making an affine change of coordinate before or after the computation gives the same result.
Affine transformations consist of translations, scalings, rotations, and shearing.
Equivariance with respect to translation means indifference with respect to the choice of origin.
Equivariance with respect to scalings means indifference with respect to the choice of units, a fundamental property from a physicist's point of view.
Not all methods are affine equivariant. For instance, \emph{partitioned} Runge--Kutta are not affine equivariant because they are not equivariant with respect to all rotations.
Partitioned Runge--Kutta methods are instead equivariant with respect to a smaller group, the partition group, see \autoref{sec:affine}; in particular, they are still equivariant with respect to translations and scalings.

We also introduce in \autoref{sec:locality} the concept of \emph{locality} of a method, and the corresponding notion for modified vector fields.
Locality means that the method exactly preserves non-isolated fixed points of the flow.
One can readily check that Runge--Kutta methods, and variants such as partitioned Runge--Kutta methods, are indeed local.

Lastly, we make the observation that a B-series can be regarded as a Taylor development in infinite dimensional space of vector fields, which contrasts to the traditional view of a Taylor expansion with respect to the time step.

Besides these new points of view, the main result relies on two classical theorems; the Peetre theorem of functional analysis and the
invariant tensor theorem of representation theory. 
Peetre's theorem  guarantees that local, multilinear maps actually depend only on the jet-space coordinates, up to a finite order, at the point at hand.
The invariant tensor theorem guarantees that equivariant multilinear maps decompose into a $\RR$-linear sum of permutations. 
This yields an equivalence between equivariant Taylor series and aromatic B-series represented in terms of directed graphs.

\subsection{Outline of the Paper}

In \autoref{sec:mainresults} we spell out the main results of the paper, as well as some of the main definitions and notations used throughout the paper.
In \autoref{sec:contagion} we show how the study of Taylor terms of local, equivariant mapping reduces to the study of symmetric, multilinear, local, equivariant mappings.
The essential argument is that if a mapping is local and equivariant, then so are all its derivatives at zero.
In \autoref{sec:extension}, we use multilinearity to further simplify the description of symmetric, multilinear, local, equivariant maps, by appealing to Peetre's theorem.
Roughly speaking, it allows us to say that a local, multilinear map depends in fact on the Taylor development of the section at hand.
In \autoref{sec:invsubspaces} we decompose the spaces of symmetric, multilinear maps defined on the jet space at a point, into invariant subspaces.
In \autoref{sec:equivarianthomogeneous}, we reap the results of the preceding sections in the homogeneous space case.
Our main result in that section is that under some mild assumptions, the Taylor development of a local and equivariant mapping is finite dimensional.
In \autoref{sec:affineequivariant}, in the affine case we are able to refine our description of those Taylor terms using aromatic trees, which are a generalisation of trees used to describe B-series.
We conclude in \autoref{sec:conclusion} by some remarks on future work.

\section{Main Result}
\label{sec:mainresults}

\subsection{Exponential Mapping}

We denote the set of diffeomorphisms of a manifold $\Man$ by $\Diff$.
The space of vector fields is denoted by $\diff$.
We will also use the description of vector fields as sections of the tangent bundle $\Tan\Man$, so we will write indifferently
\begin{equation}
	\label{eq:diffsect}
	\diff = \sects{\Tan\Man}
	.
\end{equation}
The exact solution at time one of a differential equation defined by a vector $f\in\diff$ is denoted by the mapping 
\begin{equation}
\exp\colon\diff \to \Diff
.
\end{equation}
In other words, $\exp$ is defined by the property that
\begin{equation}
	\begin{cases}
		u'(x) = f(x) & \\
		u(0) = x_0
	\end{cases}
	\iff
	\exp(f)(x_0) = u(1)
\end{equation}

A numerical integrator is an approximation of the exponential, so it is another mapping
\begin{equation}
	\Meth\colon \diff \to \Diff
	.
\end{equation}

\subsection{Backward Error Analysis}

The general idea of backward error analysis is that a method applied to a vector field $f$ actually solves exactly the flow of another, modified vector field $\tilde{f}$.
It is an extremely powerful tool for the analysis of numerical integrators.
We refer to \cite[\S\,IX]{HaLuWa06} for historical notes and further references.

With respect to the exponential mapping, the backward error analysis mapping is a mapping
\begin{equation}
\meth\colon\diff \to \diff
\end{equation}
such that
\begin{equation}
	\Meth(f) = \exp\paren[\big]{\meth(f)} \qquad f \in \diff
	.
\end{equation}
We call this mapping the \demph{modified vector field} mapping for the method $\Meth$.

Note that a modified vector field generally only exists as a formal power series, and not as a mapping as above.
However, in the sequel, we will focus on properties of general smooth mappings from $\diff$ to itself, independently of its relation to a numerical integrator.
So, we prove a rigorous result on such maps, without actually proving the claims of the introduction.
This is because introducing the complete machinery of backward error analysis in this paper would bring us too far from the focus of the paper.

\subsection{Locality}
\label{sec:locality}

The function $\exp$ is \emph{local}, which means that if a vector field $f\in\diff$ vanishes in a neighbourhood of a point $x_0\in\Man$, then $\exp(f)(x_0) = x_0$.
We thus define \demph{locality} as
\begin{equation}
	\text{$f\in\diff$ vanishes on a neighbourhood of $x_0\in\Man$} \implies \Meth(f)(x_0) = x_0
	.
\end{equation}
The property above means that $\Meth$ computes the exact solution at a point in a neighbourhood of which the vector field is zero.

Locality is a very natural property for an integrator.
Many integrators are local, with the notable exception of integrators obtained by splitting.

On the backward analysis side, the definition is even simpler.
We have to define the \demph{support} of a vector field $f\in\diff$ (a particular case of \autoref{def:support}) by
\begin{equation}
	\supp(f) \coloneqq \overline{\setc[\big]{x \in \Man}{f(x) \neq 0}}
	.
\end{equation}
A function $\meth\colon\diff\to\diff$ is \demph{local} (a particular case of \autoref{def:locality}) if it is \emph{support non increasing}, that is
\begin{equation}
	\supp\paren[\big]{\meth(f)} \subset \supp(f)
	,
	\qquad
	f\in\diff
	.
\end{equation}
It is not difficult to see that if $\meth$ is local, then so is the corresponding method $\exp\circ\meth$.

\subsection{Equivariance}
\label{sec:equivariance}

Another property of the exponential map is that it is \demph{equivariant} with respect to diffeomorphisms.
This means that for any diffeomorphism $\psi \in\Diff$ we have
\begin{equation}
	\exp(\psi \act f) = \psi \circ \exp(f) \circ \psi\inv
	,
	\qquad
	f \in \diff
	,
	\quad
	\psi \in \Diff
	,
\end{equation}
where $\psi\act f$ is the natural action of diffeomorphisms on vector fields given by 
\begin{align}
	\label{eq:groupvfaction}
	(\psi\act f)(x) = \Tan\psi \paren[\big]{f(\psi\inv(x))}
	.
\end{align}
Equivariance with respect to all diffeomorphisms is too much to ask for a numerical method, since it would imply that the method computes the exact solution of all differential equations.
We thus restrict the study to a finite dimensional Lie subgroup $\symgrp \subset \Diff$, i.e., we assume that a Lie group $\symgrp$ acts on $\Man$.
Equivariance is now written as
\begin{equation}
	\label{eq:Methequiv}
	\Meth(g \act f) = g \circ \Meth(f) \circ g\inv
	,
	\qquad
	f \in \diff
	,
	\quad
	g \in \symgrp
	.
\end{equation}

That property is also easy to express on the backward error analysis side.
We say that a mapping $\meth\colon\diff\to\diff$ is \demph{$\symgrp$-equivariant} (a particular case of \autoref{def:equivariance}) if
\begin{equation}
	\meth(g\act f) = g \act \meth(f)
	,
	\qquad
	f \in \diff
	,
	\quad
	g \in \symgrp
	.
\end{equation}
Again, one can check that if $\meth$ is equivariant, then $\exp\circ\meth$ is equivariant in the sense of \eqref{eq:Methequiv}.

\subsection{Affine Spaces}
\label{sec:affine}

In this paper, we are mostly concerned with the manifold
\begin{align}
	\Man = \RR^d
\end{align}
and with groups $\symgrp$ of the form
\begin{align}
	\symgrp = \isogrp \ltimes \RR^d
\end{align}
where $\isogrp$ is a subgroup of $\GL$.

Some examples of choices for $\isogrp$ and the corresponding group $\symgrp$ are the following.
\begin{center}
\begin{tabular}{lll}
	\toprule
	Name &  $\isogrp$ & $\symgrp = \isogrp\ltimes\RR^d$\\
	\midrule
	Affine & $\GL$ & $\Aff$ \\
	Special Affine & $\mathsf{SL}(d)$ & \\ 
	Euclidean & $\SO$ & $\SE$\\
	Galilean & $\group{O}(d-1) \ltimes \RR^{d-1}$ &  \\
	Poincaré & $\group{O}(d-1, 1)$ & \\
	Partition & $\GL[d_1]\times\cdots\times\GL[d_n]$ & \\
	Translation & $\one$ & $\RR^d$\\
	\bottomrule
\end{tabular}
\end{center}

The action of an element $(\mat,\trans) \in\symgrp$ on $\RR^d$ is simply
\begin{align}
(\mat, \trans)\cdot x \coloneqq \mat x + \trans
,
\end{align}
where $\mat\in\isogrp\subset\GL$ and $\trans\in\RR^d$.
The induced action of $\symgrp$ on a vector field $f\in\diff[\RR^d]$ given by \eqref{eq:groupvfaction} is in this case
\begin{align}
	\label{eq:affinevecact}
	\paren[\big]{(\mat,\trans) \cdot f}(x) \coloneqq \mat f \paren[\big]{\mat\inv (x - \trans)}
.
\end{align}

\subsection{Aromatic B-series: a Generalisation of B-series}
\label{sec:introaffequiseries}

A particular example we have in mind is the case $\isogrp = \GL$. It means that $\symgrp$ is the full group of affine transformations, that is $\symgrp = \Aff$ and $\Man = \RR^d$.
We observe the following fact, which is obtained by inspection using \eqref{eq:affinevecact}:
\begin{proposition}
	Runge--Kutta methods are local and affine equivariant.
\end{proposition}
%

Our main result relates the locality and affine equivariance property to the development in series, which are generalisations of B-series.
Recall that elementary differentials for B-series are encoded in trees \cite{Bu72}.
We generalize trees into aromatic trees\footnote{The nomenclature ``aromatic tree'' appears in \cite{ChMu07}, to denote what we call an ``aroma'' in  \autoref{sec:aromaticforests}.}, which are defined as follows.

An \demph{aromatic forest} is a (directed, finite) graph with at most one outgoing edge from each node.
An \demph{aromatic tree} is an aromatic forest with $n$ nodes and $n-1$ arrows.

\begin{remark}
Before we proceed further, a comment on terminology is in order.
An aromatic forest as we define it, is also called \emph{functional graph} in \cite[\S\,3.4]{Be91}.
This is, however, confusing as functional graphs are otherwise assumed to have no roots (so they really correspond to the graph of a function defined on the vertices).
Another denomination for an aromatic forest is a \emph{directed pseudo-forest} \cite{GaTa88}.

To add to the confusion, the connected components of aromatic forests with no root, which we call ``aromas'' (see \autoref{prop:aforreststruct}) have been previously considered in the geometric integration and graph theory literature with conflicting names.
In \cite{ChMu07}, they are called \emph{aromatic trees}, and in \cite{IsQuTs07}, they are called \emph{$K$-loops}.
However, that concept arose much earlier in \cite{GaTa88} under the names of \emph{directed pseudo-tree}, \emph{directed unicyclic graph} or \emph{directed 1-tree}.
The denomination ``pseudo-tree'' is also confusing as pseudo-trees have other definitions, for instance as directed graphs excluding the possibility of loops \cite[\S\,II.9]{KuMo76}.
\end{remark}


Given a vector field $f$,
to any aromatic forest $\gamma$ there corresponds an elementary differential $\eldiff(\gamma)$.
The recipe is to label each node, forming for each node with label $j$ the quantity $f^j_{i_1 i_2\ldots i_n}$ where $i_1\ldots i_n$ are the labels of the nodes pointing towards the node $j$, and finally to sum all the terms using Einstein's summation convention. The upper index on $f$ corresponds to the vector components and the lower are partial derivatives with respect to the coordinate directions, 
$f^j_{i_1 i_2\ldots i_n} = \partial^n f^j/\partial x_{i_1}\ldots\partial x_{i_n}$.

When the aromatic tree is a (regular) tree, we obtain the same elementary differential as for B-series.
An elementary example is the following tree, with four nodes and three arrows:
\begin{equation}
	\diag = \quad
	\begin{tikzpicture}[baseline=(tree1.south)]
		\begin{scope}[etree, scale=1.5]
			\placeroots{1}
			\children{child{node(i){} child{node(j){}}} child{node(m){}}}
	\end{scope}
	\node[left] at (i) {$i$};
	\node[left] at (j) {$m$};
	\node[right] at (m) {$j$};
	\node[left] at (tree1) {$k$};
\end{tikzpicture}
\quad\begin{tikzpicture}[>=stealth]\draw[->, thick](0,0) -- (1,0);\end{tikzpicture}\quad
\eldiff(\diag) = 
f^k_{ij}f^i_{m}f^mf^j
\end{equation}

Note that we do not put the restriction that the aromatic tree be a tree in the regular sense.
Here comes an example of an aromatic forest with eight nodes and seven arrows (and thus an aromatic tree) which contains a loop.

On the following example \eqref{eqex:diageldiff}, we emphasize the node labeled ``$p$'', the arrows pointing towards it, and the corresponding term in the elementary differential.
\begin{equation}
	\label{eqex:diageldiff}
	\diag = \quad 
\begin{tikzpicture}[baseline=(tree1.south)]
	\begin{scope}[etree, scale=1.5, ]
	\placeroots{4}
	\children[1]{child{node(i){}}}
	\children[2]{child{node(j){}}}
	\jointrees{2}{4}
	\begin{scope}[branch/.style={}]
	\children[3]{child{node(m){}} child{node(n){}}}
\end{scope}
	\begin{scope}[ultra thick, color=emphcolor]
	\joinmid{3}{4}
	\end{scope}
	\node[dtree node, color=emphcolor, scale=4] at (tree3){};
	\draw[emphcolor, diredge, ultra thick](tree3) -- (m);
	\draw[emphcolor, diredge, ultra thick](tree3) -- (n);
	\end{scope}
	\node[above] at (i) {$i$};
	\node[above] at (m) {$m$};
	\node[above] at (n) {$n$};
	\node[below, emphcolor] at(tree3) {$\boldsymbol{p}$};
	\node[above] at (tree4) {$q$};

	\node[left] at (tree1) {$k$};
	\node[above right] at (tree2) {$l$};
	\node[above] at (j) {$j$};
\end{tikzpicture}
\quad\begin{tikzpicture}[>=stealth]\draw[->, thick](0,0) -- (1,0);\end{tikzpicture}\quad
\eldiff(\diag) = f^k_i f^i \quad f_{jp}^l f^j \textcolor{emphcolor}{\boldsymbol{f_{mnq}^p}} f^m f^n f_l^q
\end{equation}
Clearly, the orientation of the graph matters.
We will, however, often draw the graph without orientation, the convention being that the loops are always oriented in the counterclockwise direction, the root of a tree is the bottom node, and the direction of the edges goes towards the root.
The aromatic tree in \eqref{eqex:diageldiff} will thus be written in compact form as
\begin{equation}
\begin{tikzpicture}[setree]
	\placeroots{4}
	\children[1]{child{node(i){}}}
	\children[2]{child{node(j){}}}
	\children[3]{child{node(m){}} child{node(n){}}}
	\jointrees{2}{4}
\end{tikzpicture}
\end{equation}

We now define the following generalisation of B-series.

The set of aromatic trees is denoted by $\Dit[]$, and the number of nodes of an aromatic tree $\dit$ is denoted by $\abs{\dit}$.
\begin{definition}
	For a mapping $b \colon \Dit[] \to \RR$, we define an \demph{aromatic B-series} as a formal series of the form
	\begin{equation}
	\sum_{m = 1}^{\infty} \sum_{\abs{\dit}=m} b(\dit) \eldiff(\dit)
	\end{equation}
	where $\eldiff$ is the elementary differential map.
\end{definition}

Our main result is the following.
The proof is postponed to \autoref{sec:mainproof}.
\begin{theorem}
	\label{thm:affineequiseries}
	If a smooth mapping $\meth\colon \diff[\RR^d] \to \diff[\RR^d]$ is local and affine equivariant, then its Taylor development at the zero vector field is an aromatic B-series.
\end{theorem}
This means that if a numerical method is local and affine equivariant on $\RR^d$, then its modified vector field $\meth(f)$ has a development in aromatic B-series as
\begin{equation}
	\label{eq:Taylroequi}
\meth(f) = 
b(\ATb) f 
+ b(\ATbb)f'f 
+ b(\ATbapab)f \Div(f) 
+ \cdots
\end{equation}
where $b$ is an arbitrary function from the space of aromatic trees to $\RR$.

Note that as we assume that $\meth$ is \emph{smooth} (as opposed to analytic), its Taylor series is necessarily a formal series.

We now turn to examine some examples of application of \autoref{thm:affineequiseries}.

The simplest example is the ``linear'' case.
The only aromatic forest with one node and zero arrows is clearly
\begin{equation}
	\begin{tikzpicture}[etree]
		\placeroots{1}
	\end{tikzpicture}
	.
\end{equation}
So a linear, local and equivariant map  must be proportional to the identity:
\begin{equation}
	\meth(f) = \lambda f
	,
\end{equation}
and the only non zero coefficient in \eqref{eq:Taylroequi} is $b(\ATb) = \lambda$.

The bilinear case corresponds to aromatic forests with two nodes and one arrow.
There are two possibilities, namely
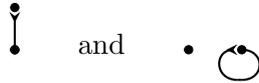
\begin{equation}
	\begin{tikzpicture}[etree]
		\placeroots{1}
		\children{child{node{}}}
	\end{tikzpicture}
	\qquad
	\text{and}
	\qquad
	\begin{tikzpicture}[etree]
		\placeroots{2}
		\jointrees{2}{2}
	\end{tikzpicture}
\end{equation}
As a result, bilinear, local equivariant maps must be of the form
\begin{equation}
	\func(f) = \lambda_1 f^i_j f^j + \lambda_2 f^i f^j_j
	,
\end{equation}
and the only nonzero coefficients in \eqref{eq:Taylroequi} are $b(\ATbb) = \lambda_1$ and $b(\ATbapab) = \lambda_2$.


We provide a table with the aromatic trees of order up to four in \autoref{tab:dtrees}.

\newcommand*\midruleheight{1pt}
\newcommand*\midmidrule{\midrule[\midruleheight]}

\newcommand*\dec[1]{\textcolor{gray}{#1}}

\begin{longtable}{l|l|l|l|l}
	\toprule
	$\diag$ & $F(\diag)$ & $\abs{\prt_{\diag}}$ & $\prt_{\diag}$ & $\prt_{\diag}'$ \\
	\midrule
	\endhead
\begin{tikzpicture}[etree]
	\placeroots{1}[.3]
\end{tikzpicture}
&
$f^k = f$
&
1
&
$(1)$
&
$(0)$
\\
\midmidrule

\begin{tikzpicture}[etree]
	\placeroots{1}[.3]
	\children[1]{child{node{}}}
\end{tikzpicture}
&
$f^k_i f^i = f'f$
&
2
&
$(1,1)$
&
$(0,1)$
\\
\begin{tikzpicture}[etree]
	\placeroots{2}[.3]
	\jointrees{2}{2}
\end{tikzpicture}
&
$f^k  f_i^i = f \Div(f)$
&
&
$\dec{(1,0) + (0,1)}$
&
\\
\midmidrule
\begin{tikzpicture}[etree]
	\placeroots{1}[.3]
	\children[1]{child{node{}} child{node{}}}
\end{tikzpicture}
&
$ f^k_{ij}f^i f^j= f''(f,f) $
&
3
&
$(2,0,1)$
&
$(0,0,2)$
\\
\begin{tikzpicture}[etree]
	\placeroots{2}
	\children[2]{child{node{}}}
	\jointrees{2}{2}
\end{tikzpicture}
&
$f^k  f^i_{ij}f^j = f \pairing{\dd \Div(f)}{f}$
&&
$\dec{(1,0,0) + (1,0,1)}$
&
\\
\midrule
\begin{tikzpicture}[etree]
	\placeroots{1}[.3]
	\children[1]{child{node{} child{node{}}}}
\end{tikzpicture}
&
$f^k_i f^i_j f^j = f'f'f$
&3&$(1,2)$& $(0,2)$
\\
\begin{tikzpicture}[etree]
	\placeroots{2}[.3]
	\children[1]{child{node{}}}
	\joinlast{2}{2}
\end{tikzpicture}
&
$f^k_i f^i f_j^j = f' f \Div(f)$
&&
$\dec{(1,1)+(0,1)}$
&
\\
\begin{tikzpicture}[etree]
	\placeroots{3}[.3]
	\jointrees{2}{3}
\end{tikzpicture}
&
$f^k f_{i}^jf_{j}^i  $
&&
$\dec{(1)+(0,2)}$
&
\\
\begin{tikzpicture}[etree]
	\placeroots{3}[.3]
	\joinlast{2}{2}
	\joinlast{3}{3}
\end{tikzpicture}
&
$f^k f_i^i f_j^j = f \paren[\big]{\Div(f)}^2$
&&
$\dec{(1)+(0,1)+(0,1)}$
&
\\
\midmidrule
\begin{tikzpicture}[setree]
	\placeroots{1}[.3]
	\children[1]{child{node {} child{node{} child{node{}}}}}
\end{tikzpicture}
&
&
4
&
$(1,3)$
&
$(0,3)$
\\
\begin{tikzpicture}[setree]
	\placeroots{2}[.3]
\children[1]{child{node {} child{node{}}}}
\jointrees{2}{2}
\end{tikzpicture}
&
&&
$\dec{(1,2)+(0,1)}$
&
\\
\begin{tikzpicture}[setree]
	\placeroots{3}
	\children[1]{child{node{}}}
	\jointrees{2}{3}
\end{tikzpicture}
&
&&
$\dec{(1,1)+(0,2)}$
&
\\
\begin{tikzpicture}[setree]
	\placeroots{3}[.35]
	\children[1]{child{node{}}}
	\jointrees{2}{2}
	\jointrees{3}{3}
\end{tikzpicture}
&
&&
$\dec{(1,1)+(0,1)+(0,1)}$
&
\\
\begin{tikzpicture}[setree]
	\placeroots{4}
	\jointrees{2}{4}
\end{tikzpicture}
&
&&
$\dec{(1)+(0,3)}$
&
\\
\begin{tikzpicture}[setree]
	\placeroots{4}
	\jointrees{2}{3}
	\jointrees{4}{4}
\end{tikzpicture}
&
&&
$\dec{(1)+(0,2)+(0,1)}$
&
\\
\begin{tikzpicture}[setree]
	\placeroots{4}
	\jointrees{2}{2}
	\jointrees{3}{3}
	\jointrees{4}{4}
\end{tikzpicture}
&
&&
$\dec{(1)+(0,1)+(0,1)+(0,1)}$
&
\\
\midrule
\begin{tikzpicture}[setree]
	\placeroots{1}
	\children[1]{child{node{} child{node{}} child{node{}}}}
\end{tikzpicture}
&
&
4
&
(2,1,1)
&
(0,1,2)
\\
\begin{tikzpicture}[setree]
	\placeroots{1}
	\children[1]{child{node {} child{node{}}} child{node{}}}
\end{tikzpicture}
&
&&
$\dec{(2,1,1)}$
&
\\
\begin{tikzpicture}[setree]
	\placeroots{2}
	\children[1]{child{node{}} child{node{}}}
	\jointrees{2}{2}
\end{tikzpicture}
&
&
&
$\dec{(2,0,1)+(0,1)}$
&
\\
\begin{tikzpicture}[setree]
	\placeroots{2}
	\children[1]{child{node{}}}
	\children[2]{child{node{}}}
	\jointrees{2}{2}
\end{tikzpicture}
&
&&
$\dec{(1,1)+(1,0,1)}$
&
\\
\begin{tikzpicture}[setree]
	\placeroots{2}
	\children[2]{child{node{} child{node{}}}}
	\jointrees{2}{2}
\end{tikzpicture}
&
&&
$\dec{(1)+(1,1,1)}$
&
\\
\begin{tikzpicture}[setree]
	\placeroots{3}
	\children[2]{child{node{}}}
	\jointrees{2}{3}
\end{tikzpicture}
&
&&
$\dec{(1)+(1,1,1)}$
&
\\
\begin{tikzpicture}[setree]
	\placeroots{3}
	\children[2]{child{node{}}}
	\jointrees{2}{2}
	\jointrees{3}{3}
\end{tikzpicture}
&
&&
$\dec{(1)+(1,0,1)+(0,1)}$
&
\\
\midrule
\begin{tikzpicture}[setree]
	\placeroots{1}
	\children[1]{child{node{}} child{node{}} child{node{}}}
\end{tikzpicture}
&
&
4
&
(3,0,0,1)
&
(0,0,0,3)
\\
\begin{tikzpicture}[setree]
	\placeroots{2}
	\children[2]{child{node{}} child{node{}}}
	\jointrees{2}{2}
\end{tikzpicture}
&
&&
$\dec{(1)+(2,0,0,1)}$
&
\\
\bottomrule
	\caption[asdf]{
		The aromatic trees of up to order four, i.e., aromatic forests with $n$ nodes and $n-1$ arrows, with $n \leq 4$.
		The composition $\prt_{\diag}$ of an aromatic tree $\diag$ is defined in \eqref{eq:defdiagcomposition}.
In short, the elements of the composition $\prt(j)$ counts the
    number of nodes with $j$ parents and $\abs{\prt} \coloneqq \sum_j \prt(j)$ is then the total number of nodes. 
    Note that non-isomorphic graphs may have the same composition, 
    e.g. there are two trees and seven aromatic trees with composition $\prt=(2,1,1)$.
	Similarly, the derived composition $\prt'$ defined by $\prt'(j) = j \kappa(j)$ is used to count the number of edges.
    Indeed, $\prt'(j)$ is the number of edges having target nodes with $j$ parents and $\abs{\prt'}$ counts the total number of edges.
	For instance, if $\prt = (2,1,1)$, then $\prt' = (0,1,2)$ and we count $\abs{\prt} = 4$, $\abs{\prt'} = 3$ which gives four nodes and three edges.
	}
	\label{tab:dtrees}
\end{longtable}

\begin{equation}
	\,
\end{equation}

We finish by a remark on the assumptions of \autoref{thm:affineequiseries}.
It is clear that equivariance plays a key role in the form of the terms in aromatic B-series.
What about the locality assumption?
One can find equivariant, non local maps which have no aromatic B-series using fractional derivatives.
For instance, in dimension one, one can define the half derivative $f^{(1/2)}$ of a smooth function from $\RR$ to $\RR$ using the square root as a Fourier multiplier, as in \cite{FoVe09}.
The mapping $f \mapsto \paren[\big]{f^{(1/2)}}^2$ is then equivariant, but it has no aromatic B-series, because it is not local.

\subsection{Pedestrian guide to the proof}

The rest of the article is essentially a proof of the claims made in \autoref{sec:mainresults}.
The bulk of the intermediary results do not pertain to numerical analysis, but rather to differential geometry and group representation theory.
The reader accustomed to numerical analysis papers may thus find the rest of the paper difficult to read.
We therefore present a simplified guide to the main ingredients of the proof \autoref{thm:affineequiseries}, our main result.
We have to forewarn the reader that this section is approximative and hides a number of technical difficulties which are properly addressed in the rest of the paper.

The proof essentially decomposes in four steps.

\subsubsection*{First step: transfer argument}

The statement in \autoref{thm:affineequiseries} is about the Taylor development of a function having certain properties: locality and affine equivariance.
The idea of the transfer argument is that those properties happen to \emph{transfer} to every component of the Taylor expansion.
Consider a Taylor development as a sequence of multilinear maps which approximate the nonlinear map $\func$.
The main objective of \autoref{sec:contagion} is \autoref{prop:Taylorequilocal}, which shows that if $\func$ is local and affine equivariant, then so are all its multilinear approximations at zero.
This reduces the proof of \autoref{thm:affineequiseries} to proving statements about multilinear maps instead of nonlinear maps.

To fix the idea, we assume henceforth that $\varphi$ is \emph{bilinear}.
The remaining of this section is now to prove that $\func(f)$ must be of the form
\begin{align}
	\varphi(f)(x) = \lambda f'(x) f(x) + \mu f(x) \Div (f)(x)
\end{align}
for some scalars $\lambda$ and $\mu$.

Note the strength of our claim here.
The space of bilinear maps is clearly infinite.
We claim that if the locality and affine equivariance are enforced, then that space reduces to a \emph{finite dimensional} space, which has in fact exactly dimension two (hence the two scalar coefficients $\lambda$ and $\mu$).
In general, the maximal dimension of that space is counted by the aromatic trees with a given number of nodes.

\subsubsection*{Second step: using locality}

The first property we are going to use is \emph{locality} of $\func$.
For technical reasons, this is not exactly how we do it in the paper, but it boils down to the same.
The main ingredient is the multilinear version of Peetre's theorem \cite[\S\,19.9]{KoMiSl93}, which says that for \emph{multilinear}, local function $\func$, it must factor through the Taylor development of $f$.
If we fix a basis of $\RR^d$, the Taylor development of $f$ is the sequence of all the derivatives of all it's components:
\begin{align}
(f^1,\ldots,f^i,\ldots,f_1^1,\ldots,f_i^j,\ldots,f^1_{11},\ldots,f^{i}_{jk},\ldots)
\end{align}
So Peetre's theorem tells us that $\func$, being local, must be of the form
\begin{align}
	\func(f)(x) = \lfunc[x](f^i(x),f^i_{j}(x),f^i_{jk}(x),\ldots)
\end{align}
Peetre's theorem moreover says that on any neighbourhood of $x$, the function $\lfunc[x]$ depends on a finite number of derivatives of $f$.

\subsubsection*{Third step: extension principle}

We now observe that since $\func$ is equivariant with respect to a \emph{transitive} action, the functions $\lfunc[x]$ for any point $x\in\RR^d$ are in fact determined by one of them at a given, arbitrary point.
We name this point $\origin$, and the proof henceforth reduces to the study of $\lfunc[\origin]$.
Moreover, the extension principles also states that the function $\lfunc[\origin]$ has to be $\GL[d]$ invariant.
The precise statement of this and the previous step is \autoref{prop:extension}.

\subsubsection*{Fourth step: invariant tensor theorem}

Finally, as we observe in \autoref{prop:prolongaction}, the action of $\GL[d]$ is precisely the tensors product action.
This means that we have reduced the original problem to a \emph{purely algebraic problem}.
In our case, it boils down to the following question: which are the $\GL[d]$-equivariant bilinear forms depending of a vector, a matrix, and higher order tensors?

The first step, which is a part of the invariant tensor theorem, is to realize that, by a scaling argument, $\lfunc[]$ can only depend on $f^i$ and $f^i_j$.
To make the presentation more transparent, we call $V$ the vector of components $f^i$, and $M$ the matrix of components $f^i_j$.
Note that the action of $\GL[d]$ on the Taylor expansion of $f$, i.e., the prolonged action of $\GL$ is the tensor product action.
In this case, the action of an element $A\in\GL$ on $M$ and $V$ is thus simply:
\begin{align}
	\label{eq:tensactex}
A \cdot V = A\, V \qquad A \cdot M = A \, M \, A\inv
\end{align}
The question is now: which bilinear maps of the form $\lfunc(V,M)$ are $\GL[d]$ equivariant?
Note that equivariance of $\lfunc$ reduces to the equation:
\begin{align}
	\lfunc(A\cdot V, A \cdot M) = A \cdot \lfunc[](V,M) \qquad A \in \GL
	,
\end{align}
which in view of \eqref{eq:tensactex} reduces to
\begin{align}
	\label{eq:matvecequi}
	\lfunc(A\, V, A \, M \, A\inv) = A\,  \lfunc(V,M) \qquad A \in \GL
	.
\end{align}

The invariant tensor theorem provides us with the answer.
Notice first that the two equivariant combinations, matrix-vector product $M \, V$ and  $V \operatorname{Tr}(M)$, are equivariant.
Indeed, the reader should verify that both expression fulfill the constraint \eqref{eq:matvecequi}.
But are these the only equivariant expressions?
In our case, this reduces to look for equivariant linear maps from $\RR^d \otimes \RR^d$ to itself.
The invariant tensor theorem tells us that such equivariant maps are obtained by permutations of the indices.
As we shall see in \autoref{thm:diagram}, we can in a natural way associate an aromatic tree to any such permutation.

We thus obtain that $\lfunc[](V,M) = \lambda \, M \, V + \mu V  \operatorname{Tr}(M)$.
Translated back in components of a vector field $f$, we obtain the expected expression
\begin{align}
	\lfunc[](f) = \lambda f' f + \mu f \Div (f)
\end{align}
which is what we wanted to prove.

The rest of this paper is now a generalization and proper formulation of those ideas.

\section{Transfer Argument}
\label{sec:contagion}

The aim of this section is to restrict the study of local equivariant maps to symmetric, $m$-linear, local equivariant maps.

\subsection{Definitions}

\subsubsection{Equivariance}

Consider two Fréchet vector spaces $\Ebv$ and $\Fbv$.
In the sequel, these spaces will be the spaces of vector fields on a manifold $\Man$, later to be specialized further to the affine case $\Man = \RR^d$.

We assume that a Lie group $\symgrp$ acts on $\Ebv$ and $\Fbv$ linearly.

\begin{definition}
\label{def:equivariance}
A mapping
\begin{equation}
	\func \in \Cinf(\Ebv,\Fbv)
\end{equation}
is \demph{equivariant} if
\begin{equation}
	\func(g \act \sigma) = g \act \func(\sigma)
	,
	\qquad
	\sigma \in \Ebv
	,
	\quad g \in \symgrp
	.
\end{equation}

We denote the \demph{space of smooth $\symgrp$-equivariant functions} by $\Cinf_{\symgrp}\paren{\Ebv,\Fbv}$:
\begin{equation}
	\Cinf_{\symgrp}\paren{\Ebv, \Fbv}
	\coloneqq
	\setc[\big]{\func \in \Cinf\paren{\Ebv,\Fbv}}{\func(g\act \sigma) = g \act \func(\sigma) \quad \forall \sigma\in\Ebv,\,g\in\symgrp}
	.
\end{equation}
\end{definition}

\subsubsection{Derivative}

Recall that the derivative operator $\Der$ is defined as \cite[\S\,I.3.18]{KrMi97}
\begin{equation}
	\Der \colon \Cinf(\Ebv,\Fbv) \to \Cinf(\Ebv, \Lin(\Ebv,\Fbv))
\end{equation}
which, for an element $\func \in \Cinf(\Ebv,\Fbv)$, is defined by
\begin{equation}
	\pairing{\Der \func(\sigma)}{\tau} \coloneqq \lim_{\varepsilon \to 0}{\frac{\func(\sigma + \varepsilon \tau) - \func(\sigma)}{\varepsilon}}, \qquad \sigma,\,\tau \in \Ebv
	.
\end{equation}

If we repeatedly make the identification \cite[\S\,I.5.2]{KrMi97}
\begin{equation}
	\Lin(\Ebv, \Lin(\Ebv,\Fbv)) \simeq \Lin(\Ebv\otimes \Ebv, \Fbv)
	,
\end{equation}
then we can consider $\Der^m\func$ as a mapping from $\Ebv$ to $\Lin(\otimes^m \Ebv,\Fbv)$.
Moreover, the image of $\Der^m\func$ is symmetric, so it factorizes to an element in $\Lin(\sym^m \Ebv,\Fbv)$ \cite[\S\,I.5.11]{KrMi97}.

\begin{definition}
For any integer $m$, we define the \demph{$m$-th Taylor term} map
\begin{equation}
T^m\colon
\Cinf\paren{\Ebv, \Fbv}
\to
\Lin\paren{\sym^m \Ebv, \Fbv}
\end{equation}
defined by
\begin{equation}
T^m(\func) \coloneqq \Der^m\func(0)
.
\end{equation}
\end{definition}

\subsubsection{Locality}

Suppose that we have two vector bundles $\bndlproj{\vecbndl}$ and $\bndlproj{\vecbndl[F]}$, that is, two vector bundles over the same base manifold $\Man$.
In the sequel, these vector bundles will be the tangent bundle of some manifold $\Man$, later to be specialized in the affine case $\Man = \RR^d$.

\begin{definition}
\label{def:support}
For a section $\sigma \in \sects{\vecbndl}$, we define its \demph{support} as the closure of the set of points where $\sigma$ is non-zero:
\begin{equation}
\supp(\sigma) \coloneqq \overline{\setc[\big]{x \in \Man}{\sigma(x) \neq 0}}
	.
\end{equation}
\end{definition}

This allows to define locality as follows.
\begin{definition}
\label{def:locality}
A mapping $\func \in \Cinf\big(\sects{\vecbndl},\sects{\vecbndl[F]}\big)$ is \demph{local} if it is \emph{support non increasing} \cite[\S\,19.1]{KoMiSl93}:
\begin{equation}
	\supp\paren[\big]{\func(\sigma)} \subset \supp(\sigma)
	.
\end{equation}
We denote the \demph{space of local functions} by $\Cinfloc$:
\begin{multline}
	\Cinfloc\paren[\big]{\sects{\vecbndl}, \sects{\vecbndl[F]}}
	\coloneqq \\
	\setc[\big]{\func \in \Cinf\paren[\big]{\sects{\vecbndl}, \sects{\vecbndl[F]}}}{\supp\paren[\big]{\func(\sigma)} \subset \supp(\sigma)}
	.
\end{multline}
\end{definition}
A function $\func$ is local if it only depends on the \emph{germs} of the sections of $\vecbndl$.



\subsection{Transfer of Equivariance}

We show that the property of equivariance is preserved under derivation, and in particular, when passing to the Taylor terms.

Since $\symgrp$ acts on $\Ebv$ and $\Fbv$ linearly, it also acts linearly on $\Lin(\Ebv,\Fbv)$ by
\begin{equation}
	g \act A \coloneqq g A g\inv
	,
\end{equation}
that is
\begin{equation}
	(g \act A)(x) \coloneqq g \act \big(A(g \inv \act x)\big)
	.
\end{equation}

\begin{lemma}
	\label{lma:equivder}
	If $\func \in \Cinf(\Ebv,\Fbv)$ is equivariant, then so is $\Der \func \in\Cinf\paren[\big]{\Ebv,\Lin\paren{\Ebv,\Fbv}}$.	
\end{lemma}
\begin{proof}
	\begin{equation}
	\pairing{\Der \func(g \act \sigma)}{\tau} = \lim_{\varepsilon\to 0}{\frac{\func(g \act \sigma + \varepsilon \tau) - \func(g\act \sigma)}{\varepsilon}}
	\end{equation}
	By linearity of the action we have
\begin{equation}
	{\func(g \act \sigma + \varepsilon \tau) - \func(g\act \sigma)} = g \act (\func(\sigma + \varepsilon(g\inv \act \tau)) - \func(\sigma))
	,
\end{equation}
so we obtain
\begin{equation}
	\pairing{\Der \func(g \act \sigma)}{\tau} = \pairing{g \act \Der \func}{\tau}
	.
\end{equation}
\end{proof}

\begin{lemma}
	\label{lma:zeroinv}
	If $\func \in \Cinf(\Ebv,\Fbv)$ is equivariant, then $\func(0)$ is invariant.
\end{lemma}
\begin{proof}
	By linearity of the action of $\symgrp$ on $\Ebv$, we have $g\act 0 = 0$, so
	\begin{equation}
		(g \act \func) (0) = g \act \func(g\inv\act 0) = g \act \func (0)
		.
	\end{equation}
	By equivariance of $\func$, we obtain
	\begin{equation}
		\func(0) = g \act \func(0)
		,
		\qquad g \in \symgrp
		.
	\end{equation}
\end{proof}

\begin{proposition}
\label{prop:Taylorequi}
	For each integer $m$, the mapping $T^m$ induces a projection from $\Cinf_{\symgrp}\paren{\Ebv,\Fbv}$ onto $\Lin_{\symgrp}\paren{\sym^m\Ebv,\Fbv}$.
\end{proposition}
\begin{proof}
	Using \autoref{lma:equivder}, we know that $\Der^m\func$ is equivariant as a mapping from $\Ebv$ to $\Lin(\sym^m \Ebv, \Fbv)$.
	Using \autoref{lma:zeroinv}, we obtain that $\Der^m\func(0)$ is invariant, which is equivalent to being equivariant from $\sym^m \Ebv$ to $\Fbv$.
\end{proof}

\subsection{Transfer of Locality}

We show that the property of locality is preserved when passing to the Taylor terms.

The Taylor terms at zero $T^m(\func) = \Der^m\func(0)$ can be considered as mappings from sections of the Whitney sum of the bundle $\vecbndl$, i.e., we have
\begin{equation}
	\sym^m \sects{\vecbndl} \subset \sects{\oplus^m \vecbndl}
	.
\end{equation}
We thus regard the mapping $T^m$ as
\begin{equation}
	T^m(\func) \in \Cinf\paren[\big]{\sects{\oplus^m \vecbndl} \to \sects{\vecbndl[F]}}
	.
\end{equation}
It thus makes sense to consider the space of local $m$-linear functions as
\begin{equation}
	\Linloc\paren[\big]{\sym^m \sects{\vecbndl}, \sects{\vecbndl[F]}}
	.
\end{equation}

\begin{proposition}
\label{prop:Taylorlocal}
	For each integer $m$, the mapping $T^m$ induces a projection from
	$\Cinfloc\paren{\sects{\vecbndl},\sects{\vecbndl[F]}}$
	onto
		$\Linloc\paren{\sym^m\sects{\vecbndl},\sects{\vecbndl[F]}}$.
\end{proposition}
\begin{proof}
	For a fixed $T \coloneqq (\tau_1,\ldots,\tau_m) \in \sects{\oplus^m \vecbndl}$, we define
	\newcommand*\tphi{\widetilde{\func}}
	\begin{equation}
		\tphi_T (t_1,\ldots,t_m) \coloneqq \func(t_1\tau_1+\cdots+t_m\tau_m)
		,
		\qquad
		t_1,\ldots,t_m \in \RR
		.
	\end{equation}
	Now,  we have
	\begin{equation}
		\pairing[\big]{\Der^{m}\func(0)}{T} = \partial_{t_1}\cdots\partial_{t_n}\tphi_T(0,\ldots,0)
		.
	\end{equation}
	Suppose that $x \not\in\supp(T)$. 
	Then, in a neighbourhood of $x$, all the sections $\tau_i$ are zero.
	By locality of $\func$,  the section $\tphi_T(t_1,\ldots,t_m)$ is zero in a neighbourhood of $x$ for all values of $t_1,\ldots,t_m$.
	As a result, the section $\pairing[\big]{\Der^m\func(0)}{T}$ is zero in a neighbourhood of $x$, so $x$ is not in the support of $\pairing[\big]{\Der^m\func(0)}{T}$.
	We have thus shown that $\Der^m\func(0)$ is local.
\end{proof}

\subsection{Transfer Result}

\begin{theorem}
\label{prop:Taylorequilocal}
Suppose that a Lie group $\symgrp$ acts linearly on the space of sections of vector bundles $\bndlproj{E}$ and $\bndlproj{F}$.
	The map $T^m$ induces a projection from
		$\Cinfloc_{\symgrp}\paren{\sects{\vecbndl}, \sects{\vecbndl[F]}}$
		onto
		$\Linloc_{\symgrp}(\sym^m \sects{\vecbndl}, \sects{\vecbndl[F]})$.
\end{theorem}
\begin{proof}
	For the $\symgrp$-equivariance it is a consequence of \autoref{prop:Taylorequi} where $\Ebv = \sects{\vecbndl}$ and $\Fbv = \sects{\vecbndl[F]}$, and for the locality it is a consequence of \autoref{prop:Taylorlocal}.
\end{proof}

Clearly, since
		$\Linloc_{\symgrp}(\sym^m \sects{\vecbndl}, \sects{\vecbndl[F]}) \subset
		\Cinfloc_{\symgrp}\paren{\sects{\vecbndl}, \sects{\vecbndl[F]}}$
		we may henceforth restrict the study to the elements in $\Linloc_{\symgrp}\paren{\sym^m\sects{\vecbndl}, \sects{\vecbndl[F]}}$.

\section{Extension Principle}
\label{sec:extension}

We know from \autoref{prop:Taylorequilocal} that it suffices to restrict our attention to symmetric multilinear local equivariant maps.
Using the transitivity property of the action of $\symgrp$, we now proceed to show that it suffices to study symmetric, multilinear maps defined on the Taylor expansion of the vector field at one point.

\subsection{Jet Space}


In this section we give some basic definition about jet spaces, and we refer the reader to \cite{Sa89} or \cite[\S\,12]{KoMiSl93} for further information.

Let $\bndlproj{\vecbndl}$ be a vector bundle above the base manifold $\Man$.
Again, we aim at applying the forthcoming results to the tangent bundle $\bndlproj{\Tan\Man}$ in the sequel.

We define the $k$-th order equivalence relation at $x\in\Man$ by $\sigma \sim_x \tau$ if $\sigma$ and $\tau$ have the same Taylor development of order $k$, for two sections $\sigma$, $\tau$ in $\sects{\vecbndl}$.
We define the \demph{$k$-th jet space} $\Jetx{x}$ at $x$ to be the corresponding quotient set of that equivalence relation.
We define $\jet[x]$ to be the corresponding projection.
We thus have
\begin{equation}
	\Jetx{x} = \jet[x]\paren[\big]{\sects{\vecbndl}}
	.
\end{equation}
$\Jetx{x}$ has a natural structure of vector space.

\begin{definition}
	\label{def:orderk}
We say that a function $\func$ is \demph{of order $k$} \cite[\S\,18.16]{KoMiSl93} if it factorizes through $\jet[x]$ at every point $x \in \Man$.
So, if $\func$ is of order $k$, then for each $x\in\Man$ there exists a function 
\begin{equation}
	\lfunc[x] \colon \Jetx{x} \to \fibre{x}{\vecbndl[F]}
\end{equation}
such that
\begin{equation}
	\label{eq:locality}
	\func(\sigma)(x) = \facfunc[x]{\sigma}
	.
\end{equation}
\end{definition}

\subsection{Extension Principle}

Here we take advantage of the idea that if a quantity is equivariant, then it is uniquely defined at a point as long as it is equivariant with respect to the isotropy group at this point \cite[\S\,4.2.6]{AlGaLyVi91}.


Suppose that $\symgrp$ acts on $\vecbndl$ by vector bundle maps.
Then $\symgrp$ acts linearly on the space $\sects{\vecbndl}$ of sections of $\vecbndl$ by
\begin{equation}
	\label{eq:actsection}
	(g \act \sigma) (x) \coloneqq g \act (\sigma(g\inv\act x))
	,
	\qquad \sigma\in\sects{\vecbndl},\quad x\in \Man
	.
\end{equation}

Moreover, that action on $\sects{\vecbndl}$ induces a vector bundle action on $\bndlproj{\Jet[]}$, in such a way that
\begin{equation}
	\label{eq:jetaction}
	g\act\jet[x](\sigma) = \jet[g\act x](g\act\sigma)
	,
	\qquad
	\sigma \in \sects{\vecbndl}
	,
	\quad
	g \in \symgrp
	,
	\quad
	x \in \Man
	.
\end{equation}

Recall also that, choosing an arbitrary point $\origin\in\Man$, the isotropy group $\isogrp$ is the subgroup consisting of the group elements that fix the origin:
\begin{align}
	\label{eq:isogrp}
	\isogrp = \setc{g\in\symgrp}{g\act \origin = \origin}
	.
\end{align}

\begin{proposition}
	\label{prop:extension}
	For each integer $k$, define the mapping $\overline{J}$ defined on the subspace of functions of order $k$ (\autoref{def:orderk}) in $\Cinf\paren[\big]{\sects{\vecbndl}, \sects{\vecbndl[F]}}$ to $\Cinf\paren{\Jetx{\origin}, \fibre{\origin}{\vecbndl[F]}}$ by
	\begin{equation}
		\overline{J}(\func) \coloneqq \lfunc
		.
	\end{equation}
	That mapping $\overline{J}$ induces a mapping $J$ from the subspace of functions of order $k$ in $\func \in \Cinf_{\symgrp}\paren[\big]{\sects{\vecbndl}, \sects{\vecbndl[F]}}$ to $\Cinf_{\isogrp}\paren{\Jetx{\origin}, \fibre{\origin}{\vecbndl[F]}}$.
	 Moreover, $J$ is a linear bijection.


\end{proposition}
\begin{proof}
	We first show $\isogrp$-equivariance of the image of $\overline{J}$.
Recall that the equivariance of $\func$ means that
\begin{equation}
	\func(g\act \sigma) = g \act \func(\sigma)
	.
\end{equation}
From the assumption \eqref{eq:locality} on $\func$ we thus obtain that for all $x \in \Man$,
\begin{equation}
	\begin{split}
		\facfunc[x]{g\act\sigma} &= \func(g\act\sigma)(x) \\
							  &= (g\act \func(\sigma))(x) \\
			   &= g \act \facfunc[g\inv\act x]{\sigma}
	.
	\end{split}
\end{equation}
So, for $x = g\act\origin$ we get
\begin{equation}
	\facfunc[x]{g\act \sigma} = g \act \facfunc{\sigma}
	.
\end{equation}

In particular, for $h\in\isogrp$ this gives
\begin{equation}
	\label{eq:facfunc}
\facfunc{h\act\sigma} = h \act \facfunc{\sigma}
,
\qquad h \in \isogrp
.
\end{equation}
Using the defining property \eqref{eq:jetaction} of the action of $\symgrp$ on $\Jetx{\origin}$ we can rewrite \eqref{eq:facfunc} as
\begin{equation}
	\lfunc(h\act \mu) = h \act \lfunc(\mu)
	,
	\qquad \mu \in \Jetx{\origin},\quad h \in \isogrp
	.
\end{equation}
This shows $\isogrp$-equivariance.

Now suppose that we are given $\lfunc[]$ in $\Cinf_{\isogrp}\paren{\Jetx{\origin}, \fibre{\origin}{\vecbndl[F]}}$.
We construct the function
\begin{equation}
	\psi_{g}(\mu) \coloneqq g \act \lfunc[](g\inv \act \mu)
	,
	\qquad
	\mu \in \Jetx{g\act\origin}
	.
\end{equation}
Notice that, using the $\isogrp$-equivariance of $\lfunc[]$ we have $\psi_{gh} = \psi_{g}$ for any $h\in\isogrp$, so for $x=g\act\origin$ we can define
\begin{equation}
	\lfunc[x](\mu) \coloneqq g \act \lfunc[](g\inv \act \mu)
	,
	\qquad
	\mu \in \Jetx{x}
	.
\end{equation}
We may then define the function $\func$ by
\begin{equation}
	\func(\sigma)(x) \coloneqq \facfunc[x]{\sigma}
	,
\end{equation}
and it is straightforward to check that $\func$ is of order $k$ and equivariant.
\end{proof}

\section{Decomposition in Invariant Subspaces}
\label{sec:invsubspaces}
\subsection{Affine Spaces}

We focus now for the rest of the paper to the case of $d$-dimensional affine spaces described in \autoref{sec:affine}.

Note that when we choose the origin $\origin = 0$, the isotropy group defined in \eqref{eq:isogrp} corresponding to the group $\symgrp=\isogrp\ltimes\RR^d$ acting on $\Man=\RR^d$ is simply $\isogrp$.

We will apply the results of the previous sections to the vector bundles
\begin{equation}
	\vecbndl = \vecbndl[F] = \Tan \RR^d
	.
\end{equation}



Observe that the group action \eqref{eq:actsection} is in that case the same as \eqref{eq:groupvfaction}, and thus becomes \eqref{eq:affinevecact} in the affine space case.

The isotropy group $\isogrp$ thus acts by matrix-vector multiplication on
\begin{equation}
 \Mv = \Tan[0]\RR^d \equiv \RR^d
 .
\end{equation}



\subsection{Jet space and Taylor Expansions}

Multilinear maps which are local at a point $x\in\Man$ are automatically of order $k$ at $x$ for some integer $k$, by Peetre's Theorem.
We first need a description of the Taylor expansion at one point.

In the affine case $\Man=\RR^d$, we make the following identification
\begin{align}
	\Tan\Rd \simeq \Man \times \Mv 
\end{align}
So, vector fields, which are sections of the tangent bundles, are now regarded as functions from $\Man$ to $\Mv$.

With this coordinate choice, the jet space of order $k$ at the origin may be rewritten as
\begin{equation}
	\label{eq:jettaylor}
	\Jetx{\origin} \simeq \Ev \otimes \bigoplus_{j=0}^{k} \sym^j \Mv^*
	.
\end{equation}

The prolonged action defined in \eqref{eq:jetaction} is now especially simple.
\begin{proposition}
	\label{prop:prolongaction}
	The prolonged action of $\isogrp$ on the $\Jetx{\origin}$ with the identification \eqref{eq:jettaylor} is the tensor product action of $\isogrp$ on $\Mv$.
\end{proposition}
\begin{proof}
	For a fixed vector field $f\in\diff[\RR^d]$, considered as a mapping $\RR^d \to \RR^d$, and an element $\mat \in\GL$, define the function $\psi$ as
	\begin{align}
		\psi(x) \coloneqq \mat \cdot f(\mat\inv x)
	\end{align}
	It follows from an induction proof that
	\begin{align}
		\bracket{\psi^{(n)}(x)}{x_1,\ldots,x_n} = \mat \bracket{f^{(n)}}{x_1,\ldots,x_n}
	\end{align}
	so at the origin $x=\origin$, we have
	\begin{align}
		\psi^{(n)}(\origin) = \mat \cdot f^{(n)}(0)
	\end{align}
	with the action defined on element of a tensor product.
\end{proof}

Since the order $k$ varies, we need the \demph{symmetric algebra} \cite[\S\,7.3]{Gr67}, defined in general for a vector $V$ as
\begin{equation}
	\sym V \coloneqq \bigoplus_{j=0}^{\infty} \sym^j V
	.
\end{equation}
The infinite sum here is taken in the coproduct sense, that is, it is the vector space generated by the finite linear combination of elements in each components.

%

\begin{proposition}
	\label{prop:linjetaction}
	The mapping defined in \autoref{prop:extension} induces the linear bijection
	\label{prop:multilinjet}
		\begin{equation}
			\linloc 
			\equiv
			\Lin_{\isogrp}\paren[\big]{ \sym^m \paren{\Ev \otimes \sym \Mv^*}, \Fv }
		,
		\end{equation}
		where the action of $\isogrp$ on $\sym^m\paren{\Ev \otimes \sym \Mv^*}$ is the tensor product action.
\end{proposition}
\begin{proof}
	First, one has to check that if $\func$ is $m$-linear, then so is $\lfunc$.
	Then, using Peetre's theorem \cite[\S\,19.9]{KoMiSl93}, we know that in a neighbourhood of $\origin$, a local multilinear map is of order $k$, for some integer $k$.
	We know by \autoref{prop:extension} that $\lfunc$ is equivariant, so we obtain the result by applying \autoref{prop:prolongaction}.
\end{proof}

\subsection{Compositions}


\begin{definition}
\label{def:composition}
We define a \demph{composition} as a map $\prt\colon \NN \to \NN$ which has finite support, i.e.
\begin{equation}
	\# \setc{j\in\NN}{\prt(j) \neq 0} < \infty
	.
\end{equation}
For such a composition, we define its \demph{size} $\abs{\prt}$ by
\begin{equation}
	\abs{\prt} \coloneqq \sum_{j=0}^{\infty} \prt(j)
	.
\end{equation}
\end{definition}

Compositions and their size arise naturally because of the following general result.
\begin{lemma}
\label{lma:Greub}
For a sequence of vector spaces $V_i$, $i\in\NN$, the following holds:
\begin{equation}
	\label{eq:termgathered}
	\sym^m \paren[\Big]{\bigoplus_{i=0}^{\infty}V_i} = \bigoplus_{\abs{\prt}=m} \paren[\bigg]{\bigotimes_{j=0}^{\infty} \sym^{\prt(j)}{V_j}}
	.
\end{equation}
\end{lemma}
\begin{proof}
	The result for a finite sum of spaces $V_i$ is in \cite[\S\,7.8]{Gr67}.
	The extension to an infinite sum is straightforward.
\end{proof}
Note that the seemingly infinite tensor product appearing in \eqref{eq:termgathered} is in fact finite, since for a fixed composition $\prt$, $\prt(j)$ is non zero only for a finite number of $j\in\NN$.




	For any composition $\prt$, define the space
\begin{equation}
	\label{eq:defSym}
	\Sym \coloneqq  \tenexpr[\Ev][\Fv]
.
\end{equation}
\newcommand*\aaababbb{
	\begin{tikzpicture}[setree]
		\placeroots{1}
		\children[1]{child{node {} child{node{}}} child{node{}}}
	\end{tikzpicture}
}
The space $\Sym$ encodes the symmetries of an elementary differential with a given composition.
As an example, consider $\prt = (2,1,1)$, 
for example associated to the tree $\aaababbb$.
Such a composition means that the elementary differential consists of two elements of the form $f^I$, one element of the form $f^I_j$ and one element of the form $f^I_{jk}$.
As we expect a vector, we thus obtain terms of the form $f^If^Jf^K_{i}f^L_{jk} \partial_l$, where we have used capital letters for the upper and minor letters for lower indices, and where $\partial_l$ denots the $l$-th unit vector.
The index set ordered as $l,I,J,K,i,L,j,k \equiv l(IJ)(Ki)(L(jk))$ corresponds to coordinates on the space $\Sym[(2,1,1)] = \Mv \otimes \Mv^* \otimes \Mv^* \otimes (\Mv^* \otimes \Mv) \otimes (\Mv^*\otimes \sym^2 \Mv)$.

\begin{lemma}
	\label{lma:binomial}
	We have the decomposition
	\begin{equation}
			\Lin\paren[\big]{ \sym^m \paren{\Ev \otimes \sym \Mv^*}, \Fv }
			=
		\bigoplus_{\abs{\prt}=m}  \Sym
		.
	\end{equation}
\end{lemma}
\begin{proof}
For a fixed composition $\prt$ we define the space
\begin{equation}
	\Vec \coloneqq  \bigotimes_{j=0}^{\infty} \sym^{\prt(j)}(\Ev \otimes \sym^j M^*)
	.
\end{equation}
Using \autoref{lma:Greub} with $V_i = \Ev\otimes \sym^i\Mv^*$, and noticing that $\Ev\otimes\sym \Mv^* = \bigoplus_{i=0}^{\infty} \Ev\otimes \sym^i \Mv^*$ we obtain
\begin{equation}
	\sym^m (\Ev\otimes\sym M^*) = \bigoplus_{\abs{\prt} = m} \Vec
	.
\end{equation}

The space $\Vec$ is finite dimensional, so we have the canonical isomorphism $\Lin(\Vec, \Fv) \equiv \Fv \otimes \Vec^*$.
Moreover, since the space $\Mv$ is finite dimensional, one has $\Sym = \Fv \otimes \Vec^*$.
\end{proof}

For instance, for the composition $\prt=(3,1,2)$, we have
\begin{equation}
	\Sym[(3,1,2)] = \Fv \otimes {\sym^3 \Ev^* \otimes {\Ev^* \otimes \Mv} \otimes \sym^2 \paren{\Ev^* \otimes \sym^2 \Mv}}
	.
\end{equation}

\subsection{Decomposition in $\isogrp$-invariant subspaces}

Note that $\isogrp$ acts on the space $\Lin(\sym^m (\Ev\otimes\sym M^*), \Fv)$, and that equivariant maps are just the invariant maps under that action.
We are interested in describing explicitly this set of equivariant (and thus invariant) maps.
Suppose that a group $\isogrp$ acts on a vector space $\Sym[]$ and that for a given index set $\Prt[]$, $\Sym[]$ is decomposed in $\Sym[] = \bigoplus_{\prt \in \Prt[]} \Sym$, where $\Sym$ is preserved by $\isogrp$.
If an element $x \in \Sym[]$ is invariant, then so are its components $x_{\prt} \in \Sym[]$, and vice versa.
In other words, we have $\Sym[]^{\isogrp} = \bigoplus_{\prt \in \Prt[]} \Sym^{\isogrp}$,
so it suffices to look for invariant elements in the subspaces $\Sym$ for $\prt \in \Prt[]$.

We thus obtain:

\begin{theorem}
	\label{prop:invdecomposition}
	For every integer $m$,
	we have the bijection:
	\begin{equation}
		\label{eq:invdecomposition}
			\Lin_{\isogrp}\paren[\big]{ \sym^m \paren{\Ev \otimes \sym \Mv^*}, \Fv }
		\to
		\bigoplus_{\abs{\prt}=m}  \Sym^{\isogrp}
		.
	\end{equation}
\end{theorem}
\begin{proof}
	The result is a consequence of 
	\autoref{lma:binomial} and the observation that the action of $\isogrp$ preserves $\Sym$ for any composition $\prt$.
\end{proof}

\section{Scalings and Finiteness}
\label{sec:equivarianthomogeneous}


We now proceed to show that if the isotropy group $\isogrp$ contains scalings, then the dimension of $\linloc[\Tan\Man][\Tan\Man]$ is finite.
We first need the definition of a derived composition, which is to be understood as a device to count the ``arrows'', as explained in \autoref{sec:generalequiseries}.
\begin{definition}
\label{def:dercomposition}
For a composition $\prt$ we define its \demph{derived composition} $\prt'\colon \NN \to \NN$ by
\begin{equation}
	\prt'(j) \coloneqq j \prt(j) \qquad j\in\NN
.
\end{equation}
\end{definition}

The following results considerably limits the possible non-trivial compositions to consider, as we shall precisely see in \autoref{prop:scalingequiseries}.
\begin{proposition}
	\label{prop:scaling}
	Assume that $\isogrp$ contains the scaling group $\GL[1]$.
	If $|\prt| \neq |\prt'| + 1$ then $\Sym^{\isogrp} = 0$.
\end{proposition}
\begin{proof}
	Consider the tensor product 
	\begin{equation}
		\label{eq:deften}
	\ten^m V \coloneqq \underbrace{V \otimes \cdots \otimes V}_{m}
	.
\end{equation}
	The symmetric tensor product $\sym^m V$ can be included into $\ten^m V$ by the standard symmetrisation map \cite[\S\,24.8]{KoMiSl93}.
	Let us define
	\begin{equation}
		\label{eq:defTen}
		{\Ten} \coloneqq \tenexpr[\Mv][\Mv][\ten]
		.
	\end{equation}
	We can thus construct an injection 
	\begin{equation}
 \Sym
\hookrightarrow
{\Ten}
.
	\end{equation}

	Now, since the symmetrisation map is $\isogrp$-equivariant, we obtain an injection
	\begin{equation}
		\label{eq:invsyminj}
		 \Sym^{\isogrp} \hookrightarrow \Ten^{\isogrp}
		.
	\end{equation}
	Finally, notice that
	\begin{equation}
		{\Ten} \simeq \ten^{\abs{\prt}}\Mv^* \otimes \ten^{\abs{\prt'} + 1}\Mv
		.
	\end{equation}
	Since $\GL[1]\subset \isogrp$, we obtain by a scaling argument \cite[\S\,24.3]{KoMiSl93} that if $\abs{\prt} \neq \abs{\prt'}+1$ the only invariant tensor in ${\Ten}$ is zero, i.e., ${\Ten}^{\isogrp} = 0$, which, using \eqref{eq:invsyminj}, implies that $\Sym^{\isogrp}= 0$.
\end{proof}

\begin{theorem}
	\label{prop:scalingequiseries}
	Under the assumption of \autoref{prop:scaling}, we have for any integer $m$
	\begin{equation}
		\linloc[\Tan\Man][\Tan\Man]	
		\equiv \bigoplus_{\substack{\abs{\prt} = m\\\abs{\prt'} = m -1}}  \Sym^{\isogrp}
		.
	\end{equation}
	In particular this implies that the space $\linloc[\Tan\Man][\Tan\Man]$ has finite dimension.
\end{theorem}
\begin{proof}
	A crude estimate yields that if $\prt(j) \neq 0$ for some $j\geq m$, then $\abs{\prt'} > m$, so $\Sym^{\isogrp} = 0$ by \autoref{prop:scaling}.
	The final statement is a consequence of the fact that $\Sym$ (and hence $\Sym^{\isogrp}$) has finite dimension for any composition $\prt$.
\end{proof}

For example, for $m = 3$, we see in \autoref{tab:dtrees} there are two compositions satisfying the requirement $\abs{\prt} = \abs{\prt'}+1 = 3$, namely $(2,0,1)$ and $(1,2)$.
\autoref{prop:scalingequiseries} for $m=3$ gives thus 
\begin{equation}
	\Linloc_{\symgrp}\paren[\big]{\sym^3 \diff[\Rd], \diff[\Rd]}
		\equiv 
		\Sym[(2,0,1)]^{\isogrp} \oplus \Sym[(1,2)]^{\isogrp}
		.
\end{equation}

\section{$\GL$ invariance and Aromatic Trees}
\label{sec:affineequivariant}

In this section we assume that the isotropy group is as big as possible, namely that
\begin{equation}
	\isogrp = \GL
	.
\end{equation}

The $\GL$ invariant tensors on tensor products of $\RR^d$ are well known.
However, our task is complicated by the presence of \emph{symmetric} tensor products.
The right combinatorial tool to tackle this difficulty turns out to be aromatic forests, which are graphs with at most one outgoing arrow at each vertex.

\subsection{Graphs}

Recall that a (finite) \demph{graph} is the collection of two finite sets, $N$ (the \demph{nodes}, or \demph{vertices}) and $A$ (the \demph{arrows}), and a two maps $s,t \colon A \to N$ (the \demph{source map} and \demph{target map}, respectively).

We will draw an arrow $a$ from a node $v$ to a node $w$ in the following manner:
\begin{equation}
	\begin{tikzpicture}
		\begin{scope}[etree, scale=1.5]
		\placeroots{2}
		\joinmid{1}{2}
	\end{scope}
	\node[left] at (tree1) {$w$};
	\node[right] at (tree2) {$v$};
	\coordinate (mid) at ($(tree1)!.5!(tree2)$);
	\node[above] at (mid) {$a$};
\end{tikzpicture}
\end{equation}
We thus have for this example $N = \set{v, w}$, $A = \set{a}$, and $s(a) = v$, $t(a) = w$.

We denote the set of nodes with $j$ incoming arrows by $N_j$:
\begin{equation}
	N_j \coloneqq \setc[\big]{v\in N}{\#  \set{t^{-1}(v)} = j}
	\qquad
	j\in \NN
.
\end{equation}
To each graph $\diag = (N,A,s,t)$ we associate the composition $\prt_{\diag}$ defined by
\begin{equation}
	\label{eq:defdiagcomposition}
	\prt_{\diag}(j) \coloneqq \# N_j
	\qquad
	j \in \NN
	.
\end{equation}

Observe that we have the disjoint unions $N = \cup_{j\in\NN} N_j$ and $A = \cup_{j\in\NN}t\inv(N_j)$, and $\# t\inv\paren{N_j} = j \# N_j$ so
\begin{equation}
	\label{eq:prtnodearrow}
	\# N = \abs{\prt}
	,
	\qquad
	\# A = \abs{\prt'}
	.
\end{equation}

%

\subsection{Aromatic Forests}
\label{sec:aromaticforests}

We focus on special graphs called \demph{aromatic forests}:

\begin{definition}
	An \demph{aromatic forest} is a graph $(N,A,s,t)$ such that $s$ is injective.
	The set of aromatic forests is denoted $\Diag[]$.
\end{definition}

For an aromatic forest $(N,A,s,t)$, we call the set $N \setminus s(A)$ the set of \demph{roots}.
An aromatic forest with one root will be called an \demph{aromatic tree}.
So an aromatic forest $(N,A,s,t)$ is an aromatic tree if and only if $\# N = \# A + 1$, or in accordance with \eqref{eq:prtnodearrow}, $\abs{\prt} = \abs{\prt'} + 1$.
Here is an example of an aromatic tree, where the root (characterized by not being the source of any arrow) is emphasized
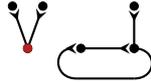
\begin{equation}
	\begin{tikzpicture}[etree]
		\placeroots{3}
		\children[1]{child{node{}} child{node{}}}
		\children[3]{child{node{}}}
		\jointrees{2}{3}
		\node[color=emphcolor] at (tree1){};
\end{tikzpicture}
\end{equation}
For that aromatic tree we have $\prt = (3,1,2)$, so $\abs{\prt} = 6$, $\prt' = (0,1,4)$, $\abs{\prt'} = 5$, and one can directly check that there are indeed 6 nodes and 5 arrows, in accordance with \eqref{eq:prtnodearrow}.


We define the set $\Diag$ of aromatic forests with composition $\prt$ by
\begin{equation}
	\Diag \coloneqq \setc{\diag \in \Diag[]}{\prt_{\diag} = \prt}
	.
\end{equation}
The set of aromatic forests with $n$ roots and composition $\prt$ is defined accordingly by
\begin{equation}
	\Diag^{n} \coloneqq \setc{\diag \in \Diag}{\abs{\prt} - \abs{\prt'} = n}
	.
\end{equation}
In particular, $\Dit$ is the set of \emph{aromatic trees} with composition $\prt$.
We will also denote by $\spangle{\Dit}$ the \emph{free vector space} generated by the trees with composition $\prt$.

\subsection{Structure of Aromatic Forests}
One defines connected components of an aromatic forest as the corresponding connected components of the undirected graph.
We define an \demph{aroma} as a connected aromatic forest with $n$ nodes and $n$ arrows.
We give the following result without proof, as we will not use it.
\begin{proposition}
	\label{prop:aforreststruct}
	An aromatic forest with $n$ roots is uniquely decomposed in $n$ regular trees, and an arbitrary number of aromas.
\end{proposition}
For example, the aromatic forest
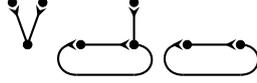
\begin{equation}
\begin{tikzpicture}[etree]
		\placeroots{5}
		\children[1]{child{node{}} child{node{}}}
		\children[3]{child{node{}}}
		\jointrees{2}{3}
		\jointrees{4}{5}
\end{tikzpicture}
\end{equation}
has one root (so it is an aromatic tree), and
can be decomposed into the regular tree $\begin{tikzpicture}[setree]\placeroots{1}\children[1]{child{node{}} child{node{}}}\end{tikzpicture}$ and the aromas $\begin{tikzpicture}[setree]\placeroots{2}\children[2]{child{node{}}}\jointrees{1}{2}\end{tikzpicture}$ and $\begin{tikzpicture}[setree]\placeroots{2}\jointrees{1}{2}\end{tikzpicture}$.

\subsection{Aromatic Trees and Invariant Tensors}

We now establish a correspondence between aromatic trees and elements of $\Sym^{\GL}$.
This is not the only result of this kind, and relations between graphs and invariant tensors are well known.
We refer to \cite{Ma08} and references therein for similar results and similar proofs.

\begin{theorem}
	\label{thm:diagram}
	Assume that $\Mv \equiv \RR^d$.
	There is a linear surjection from the space $\spangle{\Dit}$ to $\Sym^{\GL}$.
	Moreover, if $d \geq \abs{\prt}$, then that surjection is in fact a bijection.
\end{theorem}

\newcommand*\sze{\abs{\prt}}

\begin{proof}
	\mbox{  }
	\begin{enumerate}
		\item
			Suppose first that $\abs{\prt} \neq \abs{\prt'} +1$.
			On the one hand, we get from \autoref{prop:scalingequiseries} that $\Sym^{\GL} = 0$.
			On the other hand, the set $\Dit$ is empty, so $\spangle{\Dit} = 0$, and the result is thus vacuously true.
			In the remaining of the proof, we thus assume that 
			\begin{equation}
				\label{eq:treeprt}
				\abs{\prt} = \abs{\prt'} + 1
			.
		\end{equation}
		\item
			\begin{enumerate}
				\item
The goal of the proof is to obtain all the solid lines in the following commuting diagram of horizontal exact sequences.
We refer the reader to \autoref{sec:thmdiagexample} for a concrete example of the maps occurring in the diagram, which are defined later in this proof.
\begin{equation}
	\begin{tikzpicture}
		\matrix (m) [commdiag]
		{
			0 \&  K_{\otimes} \&  \Ten^{\GL} \& \Sym^{\GL} \& 0 \\
			0 \&  K_{\Sigma} \&  \spangle{\Sigma_{\sze}} \& \spangle{\Dit} \& 0 \\
		};
		\path[->]
		(m-1-1) edge  (m-1-2)
		(m-1-2) edge (m-1-3)
		(m-1-3) edge node[auto] {$\pi$} (m-1-4)
		(m-1-4) edge (m-1-5)
		(m-2-1) edge (m-2-2)
		(m-2-2) edge (m-2-3)
		(m-2-3) edge node[auto] {$\varpi$} (m-2-4)
		(m-2-4) edge (m-2-5)
		;
		\draw[->] (m-2-2) edge (m-1-2);
		\draw[->] (m-2-3) edge node[auto] {$\delta$} (m-1-3);
		\draw[-|] (m-2-3) -- (m-1-3);
		\draw[->, dashed] (m-2-4) edge node[auto] {$\eldiff$} (m-1-4);

	\end{tikzpicture}
\end{equation}
One obtain the final result by a diagram chasing argument.
First, one shows that once all the functions corresponding to solid lines are defined, then $\eldiff$ is uniquely defined.
It is also clear that if $\delta$ is surjective, then so is $\eldiff$.
\item
	By construction, the spaces $K_{\otimes}$ and $K_{\Sigma}$ are going to be isomorphic.
Since $\spangle{\Dit} = \spangle{\Sigma_{\sze}}/K_{\Sigma}$ and $\Sym^{\GL} = \Ten^{\GL}/K_{\otimes}$, a dimension argument gives that $\eldiff$ is bijective when $\delta$ is. 

We now proceed to define the spaces and maps involved in the diagram.
\end{enumerate}
		\item
	\begin{enumerate}
		\item
			Recall the definition of $\Ten$ in \eqref{eq:defTen}, using the notation \eqref{eq:deften}:
\begin{equation}
{\Ten} = \tenexpr[\Mv][\Mv][\ten]
.
\end{equation}
Using the assumption \eqref{eq:treeprt}, we make the standard identification
\begin{equation}
	\label{eq:TenLin}
	\Ten \equiv \Lin(\underbrace{\Mv\otimes\cdots\otimes\Mv}_{\abs{\prt}}, \underbrace{\Mv\otimes\cdots\otimes\Mv}_{\abs{\prt}})
		.
\end{equation}
It is tantamount to an arbitrary enumeration of the occurrences of $\Mv$ and $\Mv^*$ in the definition of $\Ten$.

We will group together the tensor products as
\begin{equation}
	\Ten = \Mv \otimes \bigotimes_{j=0}^{\infty} \bigotimes_{i=1}^{\prt(j)} T_{i}^{j}
\end{equation}
with
\begin{equation}
T_{i}^{j} \coloneqq {\Mv^* \otimes \ten^j \Mv}
.
\end{equation}

Since the numbering is arbitrary, we choose the first component to be numbered $\Mv_1$:
\begin{equation}
	\Ten = \Mv_1 \otimes \bigotimes_j \bigotimes_i T_{i}^{j} = \Lin(\ldots, \Mv_1 \otimes \cdots)
	.
\end{equation}

\item
Next, we define a mapping $\tau$ which records the relations between $\Mv$ and $\Mv^*$ in the expression of $\Ten$.
The mapping is defined as follows.
If the numbering is such that
\begin{equation}
	T_{i}^{j} =   \Mv^*_n \otimes \Mv_{n_1} \otimes \cdots \otimes \Mv_{n_j}
\end{equation}
then we define
\begin{equation}
	\tau(n_k) = n \qquad k = 1,\ldots,j
	.
\end{equation}
The mapping $\tau$ is thus defined for the integers $2,\ldots,\sze$.

\end{enumerate}
\item
	\begin{enumerate}
		\item
We denote by $\Sigma_n$ the permutation group of the set with $n$ elements.
For any element $\sigma \in \Sigma_{\sze}$ we may define a graph $(N,A,s,t)$ as follows.
We define the set of arrows $A = \set{2,\ldots,\sze}$ and nodes $N = \set{1,\ldots,\sze}$.
The target map $t$ is defined to be $\tau$, and the source map $s$ is the restriction of $\sigma\inv$ on $A$.

The graph is an aromatic forest because $s$ is injective.
Moreover, since $\# A  = \# N -1$, the aromatic forest is an aromatic tree (it is straightforward to see that the root is the node $\sigma\inv(1)$, although we will not use that observation).
Finally, one checks that the composition of the aromatic tree is $\prt$.
We thus obtain a map $\widetilde{\varpi}$ from $\Sigma_{\sze}$ to $\Dit$.

\item
Denote by $\Sigma_{A}$ and $\Sigma_{N}$ the permutation groups of the sets $A$ and $N$ respectively.
The product group $\Sigma_{A}\times\Sigma_{N}$ acts on the functions in $N^A$ by
\begin{equation}
(g_A,g_N)\cdot \xi \coloneqq g_N \circ \xi \circ g_A\inv \qquad \xi \in N^A
.
\end{equation}

We denote by $\symgrp_{\prt}$ the stabilizer subgroup of the function $\tau$, i.e.:
\begin{equation}
	\symgrp_{\prt} \coloneqq \setc{(g_A,g_N)\in \Sigma_{A}\times\Sigma_{N}}{(g_A,g_N)\cdot \tau = \tau}
	.
\end{equation}

\item
A graph isomorphism is a permutation of the nodes and arrows which is compatible with the source and arrow maps.
Permutations in $\symgrp_{\prt}$ are exactly the permutations compatible with the target map.
As a result, $\widetilde{\varpi}(\sigma) = \widetilde{\varpi}(\sigma')$ if and only $\sigma$ and $\sigma'$ are in the same orbit of the group $\symgrp_{\prt}$, i.e., if $\sigma' = (g_A,g_N) \cdot \sigma$, for some $(g_A,g_N) \in\symgrp_{\prt}$.

\item
We denote by
\begin{equation}
	\spangle{\Sigma_{\sze}}
\end{equation}
the free vector space generated by the elements of $\Sigma_{\abs{\prt}}$.

The map $\widetilde{\varpi}$ extends to a linear map $\varpi$ from $\spangle{\Sigma_{\sze}}$ to $\spangle{\Dit}$.

We may thus define the subspace $K_{\Sigma} \subset \spangle{\Sigma_{\sze}}$ which is generated by the elements $g\cdot \sigma - g'\cdot \sigma$ for $g,g' \in \symgrp_{\prt}$ and $\sigma \in \Sigma_{\prt}$.
The subspace $K_{\Sigma}$ is the kernel of the linear map $\varpi$.
\end{enumerate}
\item
	\begin{enumerate}
\item
We now  show how $\symgrp_{\prt}$ acts on $\Ten$.
On the one hand, $\Sigma_A$ acts on elementary tensors of $\ten^{\sze} \Mv$ by permuting every elements except the first one, that is
\begin{equation}
	g_A \cdot \paren[\big]{x_1 \otimes \cdots \otimes x_{\sze}} = x_1 \otimes x_{g_A(2)} \otimes\cdots \otimes x_{g_A(\sze)}
	.
\end{equation}
The permutation group $\Sigma_N$ acts on the elementary tensors $\ten^{\sze}\Mv$ by permuting all the elements.
Both actions are extended by linearity to the whole of $\ten^{\sze}\Mv$.

This induces a linear action of $\Sigma_A \times \Sigma_N$ on $\Ten$ considered as a space of mappings from $\ten^{\sze}\Mv$ to $\ten^{\sze}\Mv$ (as in \eqref{eq:TenLin}) by
\begin{equation}
	\paren[\big]{(g_A,g_N)\cdot f}(x) \coloneqq g_A \cdot f(g_N\inv \cdot x)
.
\end{equation}
In particular, this induces an action of $\symgrp_{\prt}$ on $\Ten$ as $\symgrp_{\prt}$ is a subgroup of $\Sigma_{A} \times \Sigma_{N}$.

\item
We denote by $\pi$ the projection map from $\Ten$ to $\Sym$.
By definition of $\tau$, we have $\pi(\varphi) = \pi(\varphi')$ if and only if $\varphi$ and $\varphi'$ are in the same  $\symgrp_{\prt}$ orbit.

\item
We define $K_{\otimes}$ as the subspace of $\Ten$ generated by the elements $g\cdot \varphi - g' \cdot \varphi$ for $g,g'\in\symgrp_{\prt}$ and $\varphi\in\Ten$.
$K_{\otimes}$ is thus the kernel of the linear projection map $\pi$.

Moreover, since the action of $\symgrp_{\prt}$ and $\GL$ commute, we obtain that $\pi$ projects $\Ten^{\GL}$ onto $\Sym^{\GL}$ and $K_{\otimes}$ is also the kernel of that projection (which we also denote by $\pi$).
\end{enumerate}

\item
Now, for any permutation $\sigma \in \Sigma_{\abs{\prt}}$ we associate an element $\delta(\sigma)$ of $\Ten$ as the unique linear map $\delta(\sigma)$ satisfying
\begin{equation}
	\delta(\sigma) = \paren[\big]{x_1 \otimes \cdots \otimes x_{\abs{\prt}} \mapsto x_{\sigma(1)} \otimes \cdots \otimes x_{\sigma(\abs{\prt})}}
\end{equation}
on all elementary tensors $x_1 \otimes \cdots \otimes x_{\abs{\prt}}$.

According to the invariant tensor theorem \cite[\S\,1]{Ma08}, \cite[\S\,24.4]{KoMiSl93}, \cite[\S\,3.1]{KrPr96}, \cite[Th.~2.1.4]{Fu98}, the map $\delta$ is a surjection onto $\Ten^{\GL}$, and it is a bijection if $d \geq \abs{\prt}$.


\end{enumerate}
\end{proof}

\subsection{Example}
\label{sec:thmdiagexample}

\newcommand*\ATcab{\begin{tikzpicture}[setree]
	\placeroots{1}[.3]
	\children[1]{child{node{}} child{node{}}}
\end{tikzpicture}}
\newcommand*\ATba{%
\begin{tikzpicture}[setree]
	\placeroots{2}
	\children[2]{child{node{}}}
	\jointrees{2}{2}
\end{tikzpicture}}

\newcommand*\prtm{(2,0,1)}

We choose for instance the composition $\prt = \prtm$.
The reader may want to compare with \cite{Ma08} where a similar example is studied.
We choose the numbering
\begin{equation}
	\Ten[\prtm] = \Mv_1 \otimes \Mv^*_1 \otimes \Mv^*_2 \otimes \paren{\Mv^*_3 \otimes  \Mv_2 \otimes \Mv_3}
	.
\end{equation}
The arrow set is $A = \set{2,3}$ and the node set is $N = \set{1,2,3}$.
The map $\tau$ associated to the numbering is defined by $\tau(2) = \tau(3) = 3$.
One checks that $\symgrp_{\prtm}$ is isomorphic to $\Sigma_2 \times \Sigma_2$, acting on $\Sigma_3$ by swapping the nodes $1$ and $2$, or the two arrows.
For instance, the identity permutation $(1,2,3)$ is on the one hand equivalent to the permutation $(1,3,2)$ by arrow swapping, and on the other hand equivalent to $(2,1,3)$ by node swapping.


We make the identification
\begin{equation}
	\Ten[\prtm] \equiv \Lin( \Mv_1 \otimes \Mv_2 \otimes \Lin(\Mv_2 \otimes \Mv_3, \Mv_3), \Mv_1)
,
\end{equation}
and
\begin{equation}
	\Sym[\prtm] \equiv \Lin(\sym^2 \Mv \otimes \Lin(\sym^2\Mv, \Mv), \Mv)
	.
\end{equation}
We give the results of the maps defined in the proof of \autoref{thm:diagram} in \autoref{tab:example}.

\begin{table}
	\centering
	\begin{tabular}{l|l|l|l}
	\toprule
	$\sigma \in \Sigma_3$ & $\delta(\sigma)$ & $\pi\circ\delta(\sigma)$ & $\varpi(\sigma)$ \\
	\midrule
	$(1,2,3)$ & $x \Tr\paren[\big]{\zeta(y, \cdot)}$ & \multirow{5}{*}[5pt]{$\frac{1}{2}\paren[\big]{x \Tr(\zeta(y)) + y \Tr(\zeta(x))}$}  & \multirow{5}{*}[5pt]{$\ATba$} \\
	$(2,1,3)$ & $y \Tr\paren[\big]{\zeta(x, \cdot)}$ &&\\
	$(1,3,2)$ & $x \Tr\paren[\big]{\zeta(\cdot, y)}$ &&\\
	$(3,1,2)$ & $y \Tr\paren[\big]{\zeta(\cdot, x)}$ &&\\
	\midrule
	$(2,3,1)$ & $\zeta(x,y)$ & \multirow{2}{*}[2pt]{$\zeta(x,y)$} & \multirow{2}{*}[2pt]{$\ATcab$}\\
	$(3,2,1)$ & $\zeta(y,x)$ &&\\
	\bottomrule
\end{tabular}
	\caption{
The maps $\delta$, $\pi$ and $\varpi$ are defined in the proof of \autoref{thm:diagram}.
The variables of an element of $\Ten[\prtm]$ are labeled $x,y,\zeta$, where $\zeta \in \Lin(\ten^2 \Mv, \Mv)$, and similarly for an element of $\Sym[\prtm]$.
The permutations in $\Sigma_3$ are grouped by their $\symgrp_{\prtm}$-orbits.
	}
	\label{tab:example}
\end{table}



%
%

It now follows from \autoref{tab:example} that the linear map $\eldiff$ defined in the proof of \autoref{thm:diagram} is defined on the basis of $\spangle{\Dit[(2,0,1)]}$ by
\begin{equation}
	\eldiff\paren[\big]{\ATcab} = \bracket[\Big]{x,y,\zeta \mapsto \zeta(x,y)}
,
\qquad 
\eldiff\paren[\Big]{\ATba} = \bracket[\Big]{ x,y,\zeta \mapsto  \frac{1}{2}\paren[\big]{x \Tr(\zeta(y)) + y \Tr(\zeta(x))} }
.
\end{equation}

\subsection{Degeneracies and the one-dimensional case}
\label{sec:degeneracy}

The map $\delta$ defined in the proof of \autoref{thm:diagram} is only bijective if the dimension $d$ is big enough.
This results in $\eldiff$ not being necessarily surjective either, which we see as linear dependency between elementary differentials associated to trees of a given composition.
Interestingly, it is possible to generate a basis of the kernel of $\eldiff$ for any given dimension, following \cite[\S\,II.1.3]{Fu98}.

In dimension one, it is particularly simple.
In that case, the image of $\eldiff$ in $\Sym$ has always dimension one, i.e., $\dim\paren[\big]{\eldiff(\spangle{\Dit})} = 1$.
In other words, the elementary differential of  the aromatic trees of a given composition $\prt$ all collapse to the same elementary differential.
The exact result is that for any given partition $\prt$,
\begin{equation}
	\label{eq:degeneracyoned}
	\eldiff(\gamma) = \prod_{j \in \NN} \paren[\big]{f^{(j)}}^{\prt(j)} 
	\qquad
	\forall \gamma \in \Dit[\prt] 
	\qquad 
	d = 1
	.
\end{equation}

In dimension two, one example of degeneracy is
\begin{equation}
	\eldiff\paren[\bigg]{
	\begin{tikzpicture}[setree,]
		\placeroots{3}
		\jointrees{2}{2}
		\jointrees{3}{3}
	\end{tikzpicture}
	+ 2\ 
	\begin{tikzpicture}[setree,]
		\placeroots{1}
		\children[1]{child{node{} child{node{}}}}
	\end{tikzpicture}
	-2\ 
	\begin{tikzpicture}[setree,]
		\placeroots{2}
		\children[1]{child{node{} }}
		\jointrees{2}{2}
	\end{tikzpicture}
	-
	\begin{tikzpicture}[setree,]
		\placeroots{3}
		\jointrees{2}{3}
	\end{tikzpicture}
}
	= 0 
	\qquad
	\text{if $d = 2$.}
\end{equation}

\section{Proof of the main results}

\subsection{Generalized Aromatic B-series}
\label{sec:generalequiseries}

Our main result, \autoref{thm:affineequiseries} is based on a more general result which is valid on affine spaces (see \autoref{sec:affine}) for any isotropy group $\isogrp\subset\GL$.

The following result gives rise to a generalization of aromatic B-series with any isotropy group.
Its proof is given in \autoref{sec:prooftaylorhomogeneous}.
\begin{theorem}
	\label{thm:taylorhomogeneous}
	There is a linear bijection from
\begin{equation}
	\bigoplus_{\abs{\prt} = m}  \Sym^{\isogrp}
\end{equation}
to the space of Taylor development of order $m$ at the zero vector field of local, $\symgrp$-equivariant mappings $\meth\colon \diff[\RR^d] \to \diff[\RR^d]$.
\end{theorem}

The relation with \autoref{sec:introaffequiseries} is that $\bigoplus_{\abs{\prt}=m}\Sym$ can be regarded as a generalisation of the space of aromatic forests with $m$ nodes.
The interpretation of this result is that the requirement of locality and equivariance places a considerable restriction on the possible form of the Taylor development of the modified vector field.

Note however that, in stark contrast with \autoref{thm:affineequiseries},  one cannot rule out that the dimension of each Taylor term be infinite dimensional without further assumptions on the isotropy group $\isogrp$.
Such an example is the case $\symgrp = \RR^d$ for which $\isogrp = \one$.
In that case, even the space of \emph{linear}, local equivariant maps is infinite dimensional.

However, if $\isogrp$ contains scalings, then the Taylor development of order $m$ have finite dimension.
The following refines \autoref{thm:taylorhomogeneous} if $\isogrp$ contains scalings (see \autoref{prop:scalingequiseries}):
	\begin{equation}
		\label{eq:Tenfinitesum}
	\bigoplus_{\abs{\prt} = m}  \Sym^{\isogrp}
		 =
		 \bigoplus_{\substack{\abs{\prt} = m\\\abs{\prt'} = m -1}}  \Sym^{\isogrp}
		.
	\end{equation}
	As this is a \emph{finite} sum of finite dimensional spaces, it follows from \autoref{thm:taylorhomogeneous} that the Taylor terms of any order of $\symgrp$-equivariant maps $\diff[\RR^d]\to\diff[\RR^d]$ are of \emph{finite dimension}.

	In a sense, the size $\abs{\prt}$ of a composition can be interpreted as a number of nodes, and the size $\abs{\prt'}$ can be interpreted as a number of arrows (see \eqref{eq:prtnodearrow}).
	As a result, \autoref{thm:taylorhomogeneous} along with \eqref{eq:Tenfinitesum} now associates Taylor developments of order $m$ with abstract versions of ``aromatic trees with $m$ nodes'' (i.e., ``aromatic forests with $m$ nodes and $m-1$ arrows''), which are always in finite numbers.

\subsection{General Affine Spaces}
\label{sec:prooftaylorhomogeneous}

\newcommand*\Pol{\mathcal{P}}

The \autoref{thm:taylorhomogeneous} will be proved once we prove the following equivalent result.

Recall that a symmetric $m$-multilinear map $L$ has a homogeneous version of degree $m$ that we denote $\Pol\paren{L}$, and which is obtained by repeating the argument, i.e.
\begin{equation}
	\Pol(L)(f) \coloneqq L(f,\ldots,f)
.
\end{equation}

\begin{theorem}
	\label{thm:taylorhomogeneousformula}
Given an affine space $\Man\equiv\RR^d$ with symmetry group $\symgrp = \isogrp\ltimes\RR^d$,
the space of Taylor development of order $m$ at the zero vector field of local, $\symgrp$-equivariant mappings $\meth\colon \diff[\RR^d] \to \diff[\RR^d]$ is equal to
\begin{equation}
	 \Pol \circ J\inv \paren[\Bigg]{\bigoplus_{\abs{\prt}=m} \Sym^{\isogrp}}
	,
\end{equation}
where the map $J$ is defined in \autoref{prop:multilinjet}.
\end{theorem}

\begin{proof}
Recall that we had in \eqref{eq:diffsect} the convention that $\diff = \sects{\vecbndl[\Tan\Man]}$.
Applying \autoref{prop:Taylorequilocal} with $\vecbndl = \vecbndl[F] = \Tan\RR^d$, we obtain that $T^m$ is a surjection from $\Cinfloc_{\symgrp}\paren[\big]{\diff[\Rd]}$ to $\linloc$.

By \autoref{prop:invdecomposition}, the map $J$ is a linear bijection between $\linloc$ and $\bigoplus_{\abs{\prt}=m}  \Sym^{\isogrp}$.

We thus have
\begin{equation}
	\Pol \circ T^m \paren[\Big]{\Cinfloc_{\symgrp}\paren[\big]{\diff[\Rd], \diff[\Rd]}} = \Pol \circ J\inv \paren[\Bigg]{\bigoplus_{\abs{\prt}=m} \Sym^{\isogrp}}
	.
\end{equation}
	By the standard Taylor formula \cite[\S\,5.12]{KrMi97}, $\meth$ is expanded in terms of the form $\Pol \circ T^m(\meth)$, which finishes the proof.
\end{proof}

\subsection{Affine Case}
\label{sec:mainproof}

\autoref{thm:affineequiseries} will thus be proved once we have proved the following slightly more general result.

\begin{theorem}
	Assume that $\symgrp = \Aff = \GL \ltimes \RR^d$.
	If a smooth mapping $\meth\colon \diff[\RR^d] \to \diff[\RR^d]$ is local and $\symgrp$-equivariant, then its Taylor development at order $m$ around the zero vector field is spanned by elementary differential corresponding to aromatic trees with $m$ nodes.
\end{theorem}

\begin{proof}

	Using \autoref{thm:diagram}, we have a surjection $\eldiff$ from $\spangle{\Dit}$ to $\Sym^{\GL}$.

	We thus obtain a linear surjection, also denoted by $\eldiff$, mapping $\bigoplus_{\abs{\prt}=m}\spangle{\Dit}$ onto $\bigoplus_{\abs{\prt}=m}\Sym^{\GL}$.


	Using \autoref{thm:taylorhomogeneousformula}, we obtain that $\Pol\circ J\inv \circ \eldiff$ maps the space of aromatic trees with $m$ nodes onto the space of Taylor terms of local, $\symgrp$-equivariant smooth functions.

	We conclude by observing that the elementary differential defined in \autoref{sec:introaffequiseries} is none other than the map $\Pol\circ J\inv \circ \eldiff$.

\end{proof}


\subsection{The One-dimensional Case}

In dimension one, the property observed in \eqref{eq:degeneracyoned} immediately gives the following corollary of \autoref{thm:affineequiseries}.

\begin{corollary}
	If a map from $\Cinf(\RR,\RR)$ to $\Cinf(\RR,\RR)$ is local and affine equivariant, then its Taylor development is a B-series.
\end{corollary}

%
%
%

\section{Conclusion and Outlook}
\label{sec:conclusion}
\subsection{Aromatic Runge--Kutta methods}
\label{sec:equirk}
In~\cite{Bu72} it was shown that Runge--Kutta methods are dense in the space of B-series. A natural question is to search for a class of numerical integrators which is dense in this larger space of aromatic B-series. 
As a motivation for the definition of \emph{aromatic Runge--Kutta methods}, we note that aromatic B-series may be understood as a modification of classical B-series, where the scalar field $\RR$ is replaced by the ring $\Ra$ of aromas. In a similar manner aromatic RK methods are generalized from classical RK methods by replacing the coefficients $a_{i,j}$ and $b_j$ with aromatic scalars from $\Ra$.

\begin{definition}An $s$-stage \emph{aromatic Runge--Kutta method}
for integrating $y' = f(y)$, $y(0)=y_0$ from $t=0$ to $t=1$
 is defined as
\begin{align*}
K_j & =  f\left(y_0 + \sum_{\ell=1}^s a_{j,\ell} K_\ell\right), \qquad\mbox{for $j=1,\ldots, s$,}\\
y_1 &= y_0 + \sum_{\ell=1}^s b_{\ell} K_\ell,
\end{align*}
where $a_{j,\ell}, b_\ell \in \Ra$. 
\end{definition}
For an arbitrary time step $h$, an integration step is performed by applying the above method to the scaled vector field $f\mapsto hf$. Note that the coefficients $a_{j,\ell}, b_\ell \in \Ra$ are polynomials in $h$ under this scaling. 
A simple example of an aromatic RK method is
\begin{equation}
	\label{eq:equiforwardEuler}
	\Meth(f)(x_0) = x_0 + \paren[\big]{1 + \alpha\Div(hf)(x_0)}hf(x_0),
\end{equation}
where $\alpha$ is an arbitrary real constant. 
This method is an aromatic version of  the forward Euler method, with  Butcher tableau
\begin{equation}
\begin{array}{c|c}
0 & 0 \\
\hline
& 1 + \alpha\begin{tikzpicture}[setree]\placeroots{1}\jointrees{1}{1}\end{tikzpicture} \\
\end{array}
\end{equation}

\subsection{Open problems and future research}

We have defined a generalisation of B-series in very broad terms.
We answered in particular the question of what restrictions are imposed by the locality and equivariance requirements.
This, however, raises a number of interesting questions.

Basically, every statement on B-series could be examined for the aromatic B-series counterpart.
For instance, we do not know the answer to the following questions
\begin{itemize}
	\item 
		What is the equivalent for aromatic B-series of conditions for symplecticity, or energy preservation?
	\item
		What becomes of the standard no-go theorems for Runge--Kutta methods?
\end{itemize}

For instance, we know that there are aromatic B-series that are unconditionally divergence free (which is impossible for B-series \cite{IsQuTs07}).
For instance we have

  \begin{equation}
	  \Div\paren[\Big]{
  \begin{tikzpicture}[setree, baseline]
	  \placeroots{2}[.3]
	  \children[1]{child {node{}}}
	  \joinlast{2}{2}
  \end{tikzpicture}
  + 
  \begin{tikzpicture}[setree, ]
	  \placeroots{1}[.3]
	  \children[1]{child{node{}} child{node{}}}
  \end{tikzpicture}
  -
  \begin{tikzpicture}[setree, baseline]
	  \placeroots{2}[.3]
	  \children[2]{child{node{}}}
	  \joinlast{2}{2}
  \end{tikzpicture}
  -
  \begin{tikzpicture}[setree, baseline]
	  \placeroots{3}[.3]
	  \jointrees{2}{3}
  \end{tikzpicture}
  } = 0
  .
  \end{equation}
  What's more, we are able to generate \emph{all} such divergence free expressions \cite{bicomplex}.

We also have an answer to the following fundamental question \cite{bseries}:
\begin{quote}
	\emph{What extra requirement besides locality and affine equivariance ensures that an aromatic B-series is in fact a B-series?}
\end{quote}

\section*{Acknowledgements}

The authors would like to thank Robert McLachlan for many comments and discussions.
This research was supported by the Spade~Ace~Project, 
by a Marie Curie International Research Staff Exchange Scheme Fellowship within the 7th European Community Framework Programme,
and by the J.C.~Kempe memorial fund.

\bibliographystyle{plainnat}
\bibliography{ref}

\end{document}